\documentclass[12pt]{article}

\usepackage{amsmath}
\usepackage{amsthm}
\usepackage{amssymb}
\usepackage{amsfonts}
\usepackage{epsf}
\usepackage{graphicx}
\usepackage{epstopdf}
\usepackage{hyperref}
\usepackage{color}
\newcommand{\ba}{\begin{array}}
\newcommand{\ea}{\end{array}}

\newcommand{\bg}{\begin{gathered}}
\newcommand{\eg}{\end{gathered}}

\renewcommand{\r}{\rho}

\newcommand{\z}{\zeta}

\newcommand{\bea}{\begin{eqnarray}}
\newcommand{\eea}{\end{eqnarray}}
\newcommand{\D}{\displaystyle}

\newtheorem{lem}{Lemma}
\newtheorem{thm}{Theorem}

\newtheorem{prop}{Proposition}
\theoremstyle{remark}
\newtheorem{rmk}{Remark}
\theoremstyle{definition}

 \DeclareMathOperator\Ai{{Ai}}
 \DeclareMathOperator\re{{Re}}
\DeclareMathOperator\im{{Im}} \numberwithin{equation}{section}

\newcounter{comment}
\begin{document}

\title{Asymptotics of the partition function of a Laguerre-type random matrix model}
\date{\today}
\author{Y. Zhao$^\ast$, L.H. Cao$^\dag$ and D. Dai$^\ddag$}

\maketitle

\begin{abstract}
  We study asymptotics of the partition function $Z_N$ of a Laguerre-type random matrix model when the matrix order $N$ tends to infinity. By using the Deift-Zhou steepest descent method for Riemann-Hilbert problems, we obtain an asymptotic expansion of $\log Z_N$ in powers of $N^{-2}$.
\end{abstract}

\vspace{5mm}



%%%%%%%%%%%%%%%%%%%%%%%%%%%%%%%%%%%%%%%%%%%%%%%%%%%%%%%%%%%%%%%%%%

\vspace{5mm}

\hrule width 65mm

\vspace{2mm}

\begin{description}

\item \hspace*{3.8mm}$\ast $  College  of  Mathematics  and  Computational Science, Shenzhen University, Shenzhen, Guangdong 518060, P. R. China. \\
Email: \texttt{yzhao@szu.edu.cn}

\item \hspace*{3.8mm}$\dag$  College  of  Mathematics  and  Computational Science, Shenzhen University, Shenzhen, Guangdong 518060, P. R. China. \\
Email: \texttt{macaolh@szu.edu.cn} (corresponding author)

\item \hspace*{3.8mm}$\ddag $ Department of Mathematics, City University of
Hong Kong, Hong Kong SAR. Email: \texttt{dandai@cityu.edu.hk}

\end{description}

\newpage

\section{Introduction and statement of results}

In this paper, we are interested in asymptotics of the partition function of a Laguerre-type random matrix model as follows
\begin{equation} \label{pf-def}
 Z_N(\textbf{t})=\int_0^\infty \cdots \int_0^\infty \prod_{1\leq j<k\leq
 N}(\lambda_j-\lambda_k)^2\lambda_j^\alpha
 e^{-N\sum_{j=1}^N V_\textbf{t}(\lambda_j)}d\lambda_1\cdots d\lambda_N,
 \quad\quad \alpha>-1,
\end{equation}
where $\textbf{t}:= (t_1,\cdots, t_\nu)$ and $V_\textbf{t}(\lambda)$
is a polynomial of degree $\nu$ with positive leading coefficient
\begin{equation} \label{V-exponent}
 V_\textbf{t}(\lambda):= V(\lambda; t_1,\cdots, t_\nu)=\lambda+\sum_{k=1}^\nu
 t_k \lambda^k, \qquad t_\nu >0.
\end{equation}
It is well-known that partition functions are closely related to orthogonal polynomials.
Let $p_n(x; N ,\textbf{t})$ be the $n$-th order orthonormal polynomial with respect to the weight
\begin{equation} \label{weight}
  w_N(x) := x^\alpha e^{-NV_\textbf{t}(x)} \qquad x \in (0,\infty),
\end{equation}
that is,
\begin{equation}\label{op}
 \int_0^\infty p_m(x;N,\textbf{t})p_n(x;N,\textbf{t}) x^\alpha e^{-NV_\textbf{t}(x)}dx=\delta_{m,n}, \qquad p_n(x;N,\textbf{t})=\gamma_n^{(N,\textbf{t})}x^n+\cdots,
\end{equation}
with the leading coefficient $\gamma_n^{(N,\textbf{t})}>0$. (For the sake of brevity, we
shall suppress the $N$ and $\textbf{t}$ dependence when there is no
confusion.) Then, $Z_N(\textbf{t})$ can be rewritten as
\begin{equation} \label{zn-op}
  Z_N(\textbf{t}) = N! \prod_{k=1}^{N-1} (\gamma_k^{(N,\textbf{t})})^{-2};
\end{equation}
see \cite{mehta}. With the definition in \eqref{pf-def}, it is easily seen that the logarithmic derivative of $Z_N(\textbf{t})$ with respect to the parameter $t_l$ is given by
\begin{align}
 \frac{\partial}{\partial t_l} \log (Z_N)=&\frac{1}{Z_N}\int_0^\infty \cdots \int_0^\infty (-N\sum_{j=1}^N \lambda_j^l)\prod_{1\leq j <k \leq N}(\lambda_j-\lambda_k)^2 \lambda_j^\alpha e^{-N\sum_{j=1}^N V_\textbf{t}(\lambda_j)} d\lambda_1\cdots d\lambda_N \notag \\
 =&\mathbb{E}(-N\sum_{j=1}^N \lambda_j^l)=-N^2\mathbb{E}\left(\frac{1}{N}\textrm{Tr} M^l\right), %\label{pf-diff}
\end{align}
where $\mathbb{E}$ denotes the expectation value with respect to the probability measure $d \mu_\textbf{t}$
\begin{equation}
 d \mu_{\textbf{t}}=\frac{1}{\tilde{Z}_N}\left(\det M\right)^\alpha \exp\left(-N \textrm{Tr}[V_{\textbf{t}}(M)]\right)dM.
\end{equation}
In the above formula, $dM$ is Lebesgue measure on $N\times N$
positive definite Hermitian matrices and $\tilde{Z}_N$ is the normalization constant such that $d \mu_{\textbf{t}}$
is a probability measure. Like \eqref{zn-op}, the logarithmic derivative can also be put into a form related to orthogonal polynomials as follows
\begin{equation}\label{de}
 \frac{\partial }{\partial t_l}\log Z_N=-N^2\int_0^\infty \lambda^l \rho_N^{(1)}(\lambda)d\lambda,
\end{equation}
where $\rho_N^{(1)}(\lambda)$ is the so-called ``one-point correlation function"
\begin{equation} \label{rho}
 \rho_N^{(1)}(\lambda)= \rho_N^{(1)}(\lambda;\textbf{t}):=\frac{1}{N}\lambda^\alpha e^{-NV_\textbf{t}(\lambda)}\sum_{k=0}^{N-1}p_k(\lambda)^2.
\end{equation}
By the fundamental theorem of calculus, we have from \eqref{de}
\begin{equation}\label{zn-int}
 Z_N(\textbf{t})=Z_N(\textbf{0})\exp\left[-N^2\int_{\textbf{0}}^{\textbf{t}} \int_0^\infty
 \rho_N^{(1)}(\lambda) {\nabla} V_\textbf{r}(\lambda) d\lambda  \cdot d \textbf{r} \right],
\end{equation}
where $Z_N(\textbf{0})$ is related to the classical Laguerre polynomials and given explicitly below
\begin{equation} \label{zn-laguerre}
  Z_N(\textbf{0}) = N^{-N(N+\alpha)} \prod_{j=1}^N \Gamma(j+1) \Gamma(j+\alpha);
\end{equation}
see \cite[p.321]{mehta}.

In the literature, a lot of researchers are interested in asymptotics of partition functions due to their importance in mathematical physics. For example, for the one-cut regular case in the one-matrix model, Ercolani and McLaughlin in \cite{em2003} applied the Riemann-Hilbert techniques to prove that the logarithmic of the partition function has an asymptotic expansion in powers of $1/N^2$. Later, Bleher and Its \cite{bleher2005} used another method to obtain this result. Their proof is mainly based on asymptotics of recurrence coefficients for the corresponding orthogonal polynomials. For multi-cut case in the one-matrix model, only formal asymptotic expansions of the partition functions are derived; see \cite{Bon:Dav:Eyn2000,Grava2006}. The rigorous mathematical proof is still unknown. For the one-matrix model, some people are also interested in cases when there exist some singularities in the model. For example, Krasovsky studied power-like (Fisher-Hartwig) singularities in  \cite{Kraso2007}. Moreover, the asymptotics for partition functions in multi-matrix models are studied, too. For instance, one may refer to a series of papers \cite{Col:Gui:Mau2009,Gui:Mai2005,Gui:Mau2007} done by Guionnet and her colleagues.

In this work, we plan to derive the asymptotic expansion of $\log Z_N(\textbf{t})$ as $N \to \infty$. Based on \eqref{zn-int}, it is sufficient to derive the asymptotics of the one-point correlation function $\rho_N^{(1)}(\lambda)$ for $\lambda \in [0,\infty)$. As $\rho_N^{(1)}(\lambda)$ involves orthogonal polynomials $p_k(\lambda)$, it is natural to study the asymptotics of $p_k(\lambda)$ first. This can be achieved by using the well-known Deift-Zhou steepest descent method for Riemann-Hilbert problems introduced by Deift et. al. in \cite{dkmvz,dkmvz2}; see also \cite{deift}. Note that the idea of using Riemann-Hilbert techniques was first adopted by Ercolani and McLaughlin in \cite{em2003} when they studied an Hermite-type random matrix model. Although there are quite a few parameters $t_k$ in \eqref{pf-def}, we restrict ourselves to the case when the parameter vector $\textbf{t}$ belongs to the following set $\mathbb{T}(\mathcal{T},\gamma)$
\begin{equation} %\label{T-set}
  \mathbb{T}(\mathcal{T},\gamma)=\left\{\textbf{t}\in \mathbb{R}^\nu:|\textbf{t}|\leq \mathcal{T}, t_\nu> \gamma \sum_{j=1}^{\nu-1}|t_j|\right\} \qquad \textrm{for any given $\mathcal{T}>0$, $\gamma>0$}.
\end{equation}
In fact, we will choose $\mathcal{T}$ small enough and $\gamma$ large enough such that only the one-cut case needs to be studied.

The following is our main result.

\begin{thm} \label{main-thm}
 Assume $\alpha =0$ in \eqref{weight}. There exist $\mathcal{T}>0$ and $\gamma>0$, such that for $\textbf{t} \in
 \mathbb{T}(\mathcal{T},\gamma)$, we have the following asymptotic
 expansion
\begin{equation}
\log\left(\frac{Z_N(\textbf{t})}{Z_N(\textbf{0})}\right) \sim N^2 \sum_{k=0}^\infty \frac{1}{N^{2k}} e_k(\textbf{t}), \qquad \textrm{as } N \to \infty,
\end{equation}
where $Z_N(\textbf{0})$ is given in \eqref{zn-laguerre} and $e_j(\textbf{t})$ is an
analytic function of the vector $\textbf{t}$ in a neighborhood of
$\textbf{0}$ for every $j$.
\end{thm}

\begin{rmk}
  It is natural to ask why we only consider such a simple case $\alpha = 0$ instead of general $\alpha$. The reason is that, when $\alpha \neq 0$, nice symmetric properties of the asymptotic expansion of $\rho_N^{(1)}(\lambda)$ are lost. Of course, we can still derive the asymptotic expansion of $\log\left(\frac{Z_N(\textbf{t})}{Z_N(\textbf{0})}\right)$, but this expansion is given in powers of $1/N$ instead of $1/N^2$ in the above theorem. For more detailed explanation, one may refer to Remark \ref{rmk-alpha} following Lemma \ref{lemm-s-pattern}.
\end{rmk}

As mentioned earlier, to prove our main result, we will derive the uniform asymptotic expansion of $\rho_N^{(1)}(\lambda)$ for $\lambda \in [0,\infty)$ first. Actually, with its uniform asymptotic expansion, we can get the following more general result.
\begin{thm} \label{thm2}
Let $\Theta(\lambda)$ be a $C^\infty$-smooth function and grow no faster than a polynomial for $\lambda \to \infty$.
When $\alpha =0$, there exist $\mathcal{T}>0$ and $\gamma >0$, such that for all $\textbf{t} \in
\mathbb{T}(\mathcal{T},\gamma)$, the following expansion holds true
\begin{equation} 
\int_0^\infty \Theta(\lambda)\rho_N^{(1)}(\lambda;\textbf{t})d\lambda \sim \sum_{k=0}^\infty \frac{1}{N^{2k}} \Theta_k, \qquad \textrm{as } N \to \infty,
\end{equation}
where the coefficients $\Theta_i$ depend analytically on $\textbf{t}$ for
$\textbf{t} \in \mathbb{T}(\mathcal{T},\gamma)$.
\end{thm}

This paper is organized as follows. In Section \ref{sec-rhp&measure}, we present the Riemann-Hilbert problem for the orthogonal polynomial $p_N(z)$ in \eqref{op} and calculate the equilibrium measure. In Section \ref{sec-rhp&analysis}, we give a sketch about the Deift-Zhou steepest descent analysis. Based on the uniform asymptotic expansion obtained for $p_N(z)$, we obtain the asymptotic expansion for $\rho_N^{(1)} (z)$ and show some of its nice properties in Section \ref{sec-rho&asy}. Finally, in Section \ref{sec-mainthm}, we prove Theorem \ref{thm2}, which yields Theorem \ref{main-thm} as a direct result.

%%%%%%%%%%%%%%%%%%%%%%%%%%%%%%%%%%%%%%%%%%%%%%%%%%%%%%%%%%%%%%%%%%%%%%%%%%%

\section{Riemann-Hilbert problems and the equilibrium measure} \label{sec-rhp&measure}

To obtain the asymptotics for $\rho_N^{(1)}(\lambda)$ in (\ref{rho}), it is helpful to put it into the following form by using the Christoffel-Darboux formula (see \cite{GS})
\begin{equation} \label{rho-pn}
 \rho_N^{(1)}(\lambda)=\frac{1}{N}\lambda^\alpha
 e^{-NV_\textbf{t}(\lambda)}\left[p_N'(\lambda)p_{N-1}(\lambda)-p_N(\lambda)p_{N-1}'(\lambda)\right]\frac{\gamma_{N-1}^{(N)}}{\gamma_N^{(N)}}.
\end{equation}
Then, it is obvious that the asymptotics of $\rho_N^{(1)}(\lambda)$ is determined by the asymptotics of $p_N(\lambda)$ as $N \to \infty.$
To derive the asymptotic expansion of $p_N(\lambda)$, we apply the Deift-Zhou steepest descent method for Riemann-Hilbert problems.

\subsection{Riemann-Hilbert problems}

Consider a $2\times2$ Riemann-Hilbert (RH) problem as follows:

\begin{itemize}
\item[(Y$_a)$] $Y:\mathbb{C}\setminus \mathbb{R}\rightarrow \mathbb{C}^{2\times 2}$ is analytic for $\mathbb{C}\setminus
[0,\infty)$;\vskip 2mm

\item[(Y$_b)$] $Y(z)$ possesses continuous boundary values on $(0,\infty)$. Let $Y_+(x)$ and $Y_{-}(x)$ denote the limiting value of $Y(z)$ as $z$ tends to $x$ from above and below, respectively. They satisfy
\[Y_{+}(x)=Y_{-}(x) \left(
\begin{array}{cc}
1 & x^\alpha e^{-NV_\textbf{t}(x)} \\
0 & 1
\end{array}
\right) \qquad \textrm{for $x\in (0,\infty)$; }
\]

\item[(Y$_c)$]  for $z\in \mathbb{C}\setminus [0,\infty)$,
\begin{equation} \label{Y-large}
  Y(z)=\Big[I+O\Big(\frac{1}{z}\Big)\Big] \left(
\begin{array}{cc}
z^N & 0 \\
0 & z^{-N}
\end{array}
\right) \quad\text{as}\,\,z\rightarrow \infty;
\end{equation}

\item[(Y$_d)$]  $Y(z)=O\left(
\begin{array}{cc}
1 & \eta_\alpha(z) \\
1 & \eta_\alpha(z)
\end{array}
\right)$ as $z\to 0, z\in \mathbb{C}\setminus [0,\infty)$, where
$\eta_\alpha(z)$ is defined by
\begin{equation}
 \eta_\alpha(z)=
 \begin{cases}
  1, & \text{if}\,\,\alpha>0,\\
 \displaystyle \log(1/|z|), & \text{if}\,\,\alpha=0,\\
 \displaystyle |z|^\alpha, & \text{if}\,\,-1<\alpha<0.
 \end{cases}
\end{equation}
\end{itemize}

According to the significant results of Fokas, Its and Kitaev \cite{fik}, the
solution of the above RH problem is given in terms of the monic polynomials
$\pi_N(x)$ orthogonal with respect to $w(x)$. This establishes an
important relation between orthogonal polynomials and
Riemann-Hilbert problems.

\begin{lem} \label{fik-lemma} (Fokas, Its and Kitaev \cite{fik}).
The unique solution to the above RH problem for $Y$ is given by
\begin{equation} \label{Y-sol}
 Y(z)=\left(
 \begin{array}{ll}
 \pi _N(z) &  C[\pi_N w](z) \\
 c_N\pi_{N-1}(z) & c_NC[\pi_{N-1} w](z)
 \end{array}
 \right)
\end{equation}
where $w$ is the weight function given in \eqref{weight}, $c_N=-2\pi
i\gamma_{N-1}^2$ and
$$
C[f](z):=\frac{1}{2\pi i}\int_0^{\infty}\frac{f(\zeta)}{\zeta-z}d\zeta,\quad
z\in\mathbb{C}\setminus[0,\infty),
$$
is the Cauchy transform of $f$.
\end{lem}

\subsection{Equilibrium measure}

The equilibrium measure plays an important role in the Riemann-Hilbert analysis and we wish to
calculate it explicitly. Recall that the equilibrium measure $\mu_V$ is
the unique minimizer of the following weighted energy
\begin{equation}
 I_V(\mu)=\int\int \log \frac{1}{|x-y|}d\mu(x)d\mu(y)+\int V(x)d\mu(x)
\end{equation}
among all probability measures on $[0, \infty)$. For our problem, we have the explicit formula as follows.

\begin{thm} \label{thm-measure}
There are $\mathcal{T}_0>0$ and $\gamma_0>0$ such that for all $0<\mathcal{T}<T_0$ and
$\gamma>\gamma_0$, the following holds true. If $\textbf{t} \in
\mathbb{T}(\mathcal{T},\gamma)$, then we have
\begin{equation} \label{measure-mu}
 d\mu_V=\psi_V(x)dx, \quad\quad
 \psi_V(x)=\frac{1}{2\pi}\chi_{(0,\beta)}(x)\sqrt{\frac{\beta-x}{x}}h(x),
\end{equation}
where $h(x)$ is a polynomials of degree $\nu-1$. This polynomial is strictly
positive on the interval $[0,\beta]$ and defined as
\begin{equation} 
 h(z)= \frac{1}{2\pi
 i}\oint_{\Gamma_z}\sqrt{\frac{y}{y-\beta}}V_\textbf{t}'(y)\frac{dy}{y-z}, \qquad
 \textrm{for} \quad z\in \mathbb{C}\backslash[0,\beta],
\end{equation}
where $V_\textbf{t}(y)$ is given in \eqref{V-exponent}, $\Gamma_z$ is a positively oriented contour containing $[0,\beta]$
and $z$ in its interior. The endpoint $\beta$ is determined by the
equation
\begin{equation}\label{beta-def}
\int_0^\beta
 V_\textbf{t}'(x)\sqrt{\frac{x}{\beta-x}}dx=2\pi.
\end{equation}
\end{thm}

Readers may compare the above theorem with Theorem 3.1 in
\cite{em2003} and find some similarities. The main difference is
that equilibrium measure in \eqref{measure-mu} is supported on
$[0,\beta]$ and possesses a square root singularity at 0, while the measure in \cite[Thm 3.1]{em2003} is
supported on $[\alpha,\beta]$ with $\alpha<0<\beta$ and vanishes like a square root at both endpoints. In fact, our
proof of the above theorem is based on \cite[Thm 3.1]{em2003} and a
nice relation between the equilibrium measure $\mu_V$ and the other one
$\nu_V$, which is the unique minimizer of the weighted energy
\begin{equation}
 I_{V(x^2)/2}(\nu)=\int\int \log \frac{1}{|x-y|}d\nu(x)d\nu(y)+\int \frac{V(x^2)}{2}d\nu(x)
\end{equation}
among all probability measures on $\mathbb{R}$. The relation is given in the following lemma.

\begin{lem}
Let $\psi_V$ be the density of $\mu_V$ and $\phi_V$ be the density of $\nu_V$.
We have
\begin{equation}\label{phi-psi-relation}
 \phi_V(x)=|x|\psi_V(x^2) \qquad \textrm{for } x \in \mathbb{R}.
\end{equation}
\end{lem}
\begin{proof} See \cite[Lemma 2.2]{ck2008}. \end{proof}

With the above Lemma, we can prove Theorem \ref{thm-measure}

\bigskip

\noindent\emph{Proof of Theorem \ref{thm-measure}.} Let $W_\textbf{t}(x)$ be
the following potential supported on the whole real axis
\begin{align}\label{potential-W}
 W_\textbf{t}(x)= \frac{1}{2}V_\textbf{t}(x^2)=\frac{1}{2}x^2+\sum_{k=1}^\nu \frac{t_k}{2}x^{2k} \quad\quad\quad x \in \mathbb{R}.
\end{align}
Then according to \cite[Thm. 3.1]{em2003}, its corresponding
equilibrium measure is given by
\begin{equation}
  d \nu_W= \frac{1}{2\pi}\chi_{[\tilde{\alpha},\tilde{\beta}]}(x)\sqrt{(x-\tilde{\alpha})(\tilde{\beta}-x)} \; \tilde{h}(x) dx,
\end{equation}
where $\tilde{h}(z)$ is a polynomial of degree $2\nu -2$ and defined as
\begin{equation}\label{tilde-h}
\tilde{h}(z)=\frac{1}{2\pi i} \oint_{\tilde\Gamma_z} \frac{W_\textbf{t}'(s)}{\sqrt{(s-\tilde{\alpha})(s-\tilde{\beta})}}\frac{ds}{s-z}.
\end{equation}
Here the integral is taken on a positively oriented contour $\tilde\Gamma_z$ containing $(\tilde \alpha, \tilde\beta)$ and $z$ in its interior,
the endpoints $\tilde{\alpha}$ and $\tilde\beta$ are determined by the following two integrals
\begin{equation} \label{endpoint-EM}
 \int_{\tilde{\alpha}}^{\tilde{\beta}}
 \frac{W_\textbf{t}'(s)}{\sqrt{(s-\tilde{\alpha})(\tilde{\beta}-s)}}ds=0
 \quad \textrm{and} \quad \int_{\tilde{\alpha}}^{\tilde{\beta}}
 \frac{sW_\textbf{t}'(s)}{\sqrt{(s-\tilde{\alpha})(\tilde{\beta}-s)}}ds=2\pi.
\end{equation}
Recall \eqref{potential-W}, $W_\textbf{t}'(x)$ is an odd function. Then the above formulas gives us $ \tilde{\alpha}=-\tilde{\beta}$ with $\tilde\beta>0$ and
\begin{equation}
  d \nu_W= \frac{1}{2\pi}\chi_{[-\tilde{\beta},\tilde{\beta}]}(x)\sqrt{\tilde{\beta}^2-x^2} \; \tilde{h}(x) dx.
\end{equation}
Using \eqref{phi-psi-relation}, we obtain \eqref{measure-mu} with
$\beta=\tilde{\beta}^2$ and $h(x) =\tilde{h}(\sqrt{x})$. In fact,
one can easily verify that \eqref{beta-def} and \eqref{endpoint-EM}
are consistent with the relations \eqref{potential-W}. To see that $h(x) =\tilde{h}(\sqrt{x})$ is a
polynomial in $x$ of degree $\nu -1$, we recall that $W'_\textbf{t}(s)$ is a
polynomial of odd powers. Then calculating residue at $\infty$ in
\eqref{tilde-h}, it is easily seen that $\tilde{h}(z)$ is a
polynomial of even powers. As a consequence, $h(x)
=\tilde{h}(\sqrt{x})$ is a polynomial in $x$ of degree $\nu -1$.
This finishes the proof of our theorem.  \hfill $\Box$

\section{Deift-Zhou steepest descent analysis} \label{sec-rhp&analysis}

In the standard Deift-Zhou steepest descent analysis, one introduces a sequence of
transformations:
\[Y\rightarrow U\rightarrow T,\]
such that $T(z)$ satisfies a RH problem with simplified jump conditions.
Then, when $N$ is large, some parametricies $ T^{(A)}(z)$ are constructed to
approximate $T(z)$ in different regions of the complex-$z$ plane. As the above transformations are revertible, one gets large-$N$ asymptotics for $Y(z)$ from $ T^{(A)}(z)$.

For our problem, since the equilibrium measure in \eqref{measure-mu} is supported on one interval, this is the so-called one-cut case. The Riemann-Hilbert analysis is similar to that done by Vanlessen in \cite{van2007}. So we only give a sketch for the completeness of this paper. The interested readers may refer to \cite{van2007} for details.

\bigskip

\noindent\textbf{Normalization: $Y \to U$.} The first transformation in the Deift-Zhou steepest descent analysis
is to normalize the large-$z$ behavior of $Y(z)$ in \eqref{Y-large} and make it tend to the identity matrix. To achieve it, we introduce the following $g$-function
\begin{equation} 
 g(z):=\int \log(z-x) \psi_V(x) dx = \frac{1}{2\pi} \int_0^\beta \log(z-x) \sqrt{\frac{\beta-x}{x}} \, h(x) dx , \quad  z\in \mathbb{C}\backslash(-\infty, \beta],
\end{equation}
where the principal branch of the logarithm is taken. It is obvious
that $g(z)$ behaves like $\log z$ when $z$ is large. Then, the first transformation $Y \to U$ is defined as
\begin{equation}\label{YtoU}
 U(z)=e^{-\frac{Nl_V}{2}\sigma_3}Y(z)e^{-Ng(z)\sigma_3}e^{\frac{Nl_V}{2}\sigma_3}, \quad\quad z\in \mathbb{C}\backslash\mathbb{R},
\end{equation}
such that the large-$z$ behavior of $U(z)$ is $U(z) = I + O(z^{-1})$ as $z \to \infty$. Here $\sigma_3$ is the Pauli matrix $\begin{pmatrix}
  1 &0 \\ 0 & -1
\end{pmatrix}$.

\bigskip

\noindent\textbf{Opening of the lens: $U \to T$.} In the second transformation $U \to T$, we deform the original interval $[0,\infty)$ and open lens. With this transformation, the rapidly oscillatory jump matrices for $U$ will be reduced to jump matrices who tend to $I$ exponentially except in the neighbourhood of $(0,\beta)$.

Before we introduce the transformation, we need some auxiliary functions. Define
\begin{equation}
  \varphi (z):= \frac{1}{2} \sqrt{\frac{z-\beta}{z}} h(z), \qquad  \textrm{for } z\in \mathbb{C}\backslash[0,\beta],
\end{equation}
where the principal branch of the square root is taken, and
\begin{equation} 
 \xi(z):= -  \int_\beta^z \varphi(y) dy =-\frac{1}{2} \int_\beta^z \sqrt{\frac{y-\beta}{y}}h(y) dy\quad\quad \textrm{for } z\in \mathbb{C}\backslash(-\infty,\beta],
\end{equation}
where the path of integration does not cross the real axis and the
principal branch of the square root is taken. The $\xi$-function satisfies the following properties.

\begin{prop} \label{xi-prop}
  For $x \in \mathbb{R}$, we have
  \begin{eqnarray}
    && \xi_+(x)-\xi_-(x)=2\pi i \qquad \textrm{for } x\in (-\infty, 0), \\
    && 2\xi_+(x)=-2\xi_-(x)=g_+(x)-g_-(x) \qquad \textrm{for } x\in (0,\beta), \\
    && \xi(x)<0 \qquad \textrm{for }  x \in (\beta, \infty).
  \end{eqnarray}
  Moreover, there exists a $\delta>0$ such that
  \begin{equation}
    \re \xi(z) >0\quad\quad \textrm{for } 0<|\im  z|<\delta, \quad 0<\re z <\beta.
  \end{equation}
  and
  \begin{equation} \label{g and v}
    2\xi(z)=2g(z)-V_\textbf{t}(z)-l_V, \quad\quad\textrm{for }z \in \mathbb{C} \setminus [\beta, \infty),
  \end{equation}
  where $l_V$ is a constant.
\end{prop}
\begin{proof}
  The proof is similar to the analysis in \cite[Sec. 3.4]{van2007}.
\end{proof}

The second transformation $U \to T$ is defined as
\begin{equation} \label{UtoT}
T(z):= \begin{cases}
  U(z), & \text{for $z$ outside the lens-shaped region} \\
  U(z) \left( \begin{matrix} 1 & 0 \\ -z^{-\alpha} e^{-2N \xi(z)} & 1
    \end{matrix} \right), & \text{for $z$ in the upper lens region}, \\
  U(z) \left( \begin{matrix} 1 & 0 \\ z^{-\alpha} e^{-2N\xi(z)} & 1
\end{matrix} \right), & \text{for $z$ in the lower lens region};
\end{cases}
\end{equation}
see Fig. \ref{fig-gamma}.
\begin{figure}[h]
\begin{center}
\includegraphics[width=200pt]{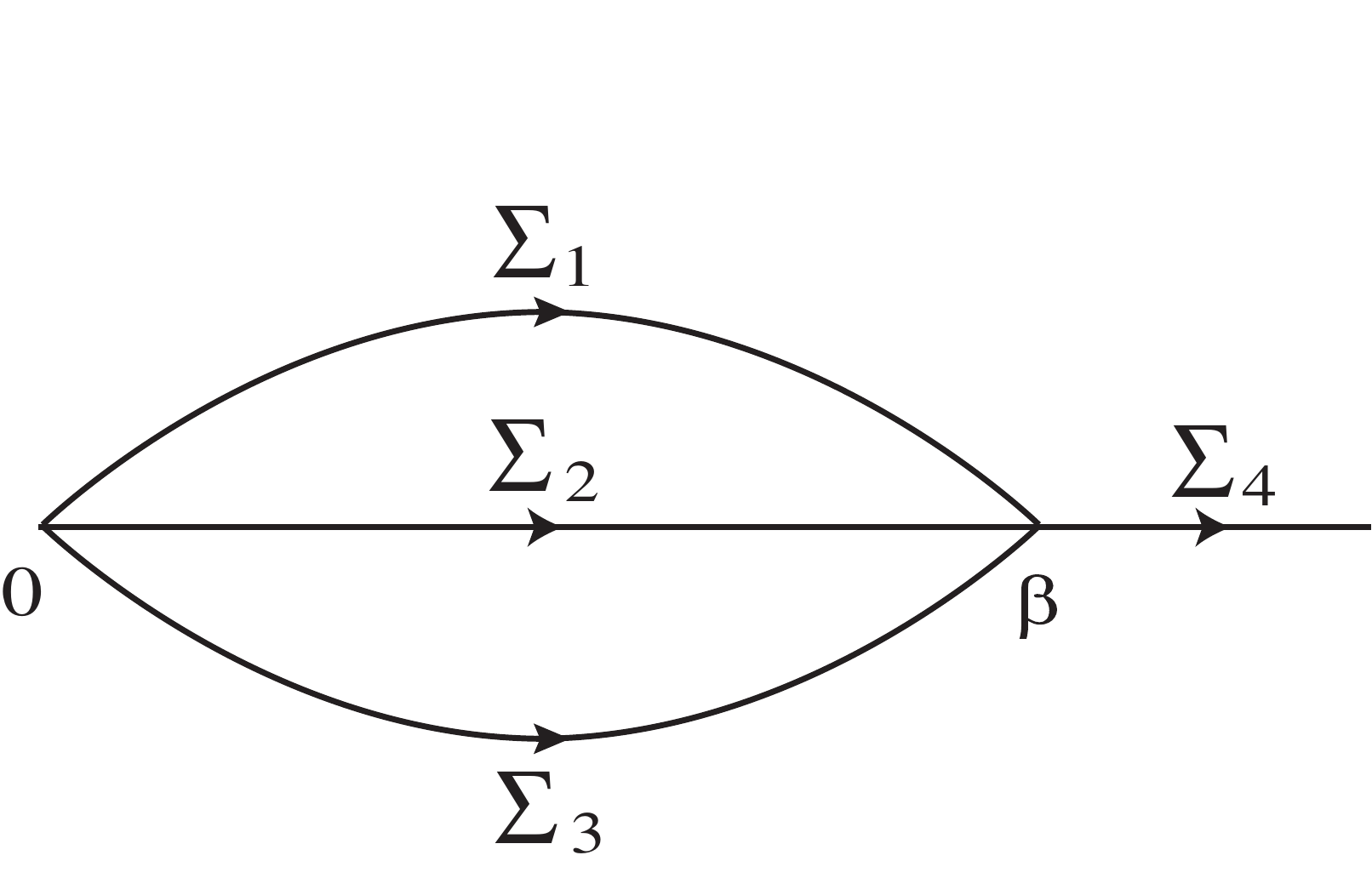}
\end{center}
\caption{ Contour $\Sigma_T= \bigcup_{i=1}^4\Sigma_i$} \label{fig-gamma}
\end{figure}
Then, $T$ satisfies a RH problem with the following jump conditions $J_T(z)$
\begin{equation} \label{jump-jt}
  J_T(z)= \begin{cases}
    \left(\begin{matrix} 1&\quad 0\\z^{-\alpha}e^{-2N\xi(z)}&\quad 1\end{matrix}\right), & z \in \Sigma_1\cup \Sigma_3 \\
    \left(\begin{matrix} 0&\quad z^\alpha\\-z^{-\alpha}&\quad 0\end{matrix}\right), & z \in (0, \beta) \\
    \left(\begin{matrix} 1&\quad z^{\alpha}e^{2N\xi(z)}\\0&\quad 1\end{matrix}\right), & z \in (\beta, \infty).
  \end{cases}
\end{equation}
Of course $T(z)$ is still normalized at $\infty,$ that is $
T(z)=I+O\left(z^{-1}\right)$ as $z\to \infty.
$

\subsection{Outside parametrix}

According to the properties of $\xi(z)$ in Proposition \ref{xi-prop}, we can choose $\Sigma_1$ and $\Sigma_3$ such that
$\re \xi(z)>0$ for $z\in \Sigma_1\cup\Sigma_3$, and $\re \xi(x)<0$ for $x\in(\beta,\infty)$. Thus, $J_T(z)\to I$ exponentially as $N\to\infty$
for $z $ bounded away from $[0,\beta]$. It suggests that, for large $N$, the solution of the RH problem for $T(z)$ behaves
asymptotically like the solution of the following RH problem for
$T^{(\infty)}(z)$:
\begin{itemize}
\item[(a)]  $T^{(\infty)}(z)$ is analytic in $\mathbb{C}\setminus
[0,\beta]$;

\item[(b)]  for $x\in (0,\beta)$,
\begin{equation} %\label{N-infty-jump}
  T^{(\infty)}_{+}(x)=T^{(\infty)}_{-}(x)  \left( \begin{matrix} 0 & x^\alpha \\
-x^{-\alpha} & 0
\end{matrix} \right);
\end{equation}

\item[(c)]  $T^{(\infty)}(z)=I+O(z^{-1})$, as $z\to\infty$.

\end{itemize}
The solution to the above RH problem is given explicitly as follows
\begin{equation} \label{out-para}
  T^{(\infty)}(z)=2^{-\alpha \sigma_3} \left(\begin{matrix}
\frac{a(z)+a^{-1}(z)}{2}&\quad
\frac{a(z)-a^{-1}(z)}{2i}\\\frac{a(z)-a^{-1}(z)}{-2i}&\quad
\frac{a(z)+a^{-1}(z)}{2}\end{matrix}\right)D(z)^{-\sigma_3},
\quad\quad z\in \mathbb{C}\backslash[0,\beta]
\end{equation}
with
\begin{equation} \label{fun-a(z)}
  a(z)=\frac{(z-\beta)^{\frac{1}{4}}}{z^{\frac{1}{4}}} \qquad \textrm{for } z\in \mathbb{C}\backslash[0,\beta]
\end{equation}
and
\begin{equation} \label{fun-d(z)}
  D(z)=\frac{z^{\frac{\alpha}{2}}}{\phi(z)^{\frac{\alpha}{2}}},
\quad\quad \phi(z)=2z-\beta+2z^\frac{1}{2}(z-\beta)^{\frac{1}{2}} \qquad \textrm{for } z\in \mathbb{C}\backslash[0,\beta],
\end{equation}
where principal branches of powers are taken in the above formulas.

Note that the above parametrix is not valid in the neighbourhoods of endpoints $0$ and $\beta$, since the jump matrices do not tend to $I$ as $N \to \infty$. So, some local parametrix constructions are needed near there two endpoints. As the equilibrium measure has different behaviors near $0$ and $\beta$ (see Theorem \ref{thm-measure}), different parametrices appear in their neighborhoods.  More precisely, we will have an Airy-type parametrix near $\beta$ since the density of equilibrium measure vanishes like a square root at $\beta$; and a Bessel-type parametrix near 0 as the density has a square root singularity there.

\subsection{Local parametrix near $\beta$}

This parametrix is constructed in terms of Airy functions as follows
\begin{equation}
 T^{(\beta)}(z)=E^{(\beta)}(z)\Psi^{(\beta)}( f(z))e^{-N\xi(z)\sigma_3}z^{-\frac{1}{2}\alpha\sigma_3},
\end{equation}
where
\begin{equation}
 E^{(\beta)}(z)=T^{(\infty)}(z) z^{\frac{1}{2}\alpha
 \sigma_3}e^{\frac{\pi i}{4}\sigma_3}
 \frac{1}{\sqrt{2}}\left(\begin{matrix} 1& -1\\1&
 1\end{matrix}\right) f(z)^{\frac{\sigma_3}{4}},
\end{equation}
\begin{equation}
f(z):=\left(\frac{3N}{2} \int_\beta^z \varphi(s) ds\right)^\frac{2}{3} =\left(\frac{3N}{4} \int_\beta^z \sqrt{\frac{s-\beta}{s}}h(s) ds\right)^\frac{2}{3}
\end{equation}
for $ z\in \mathbb{C}\backslash(-\infty,\beta],$ and
\begin{equation}
 \Psi^{(\beta)}(z)=\sqrt{2\pi}\,e^{-\frac{\pi i}{12}}
 \times \begin{cases}
    \left(\begin{matrix} \Ai(z)&\quad \Ai(\omega^2z)\\ \Ai'(z)&\quad \omega^2 \Ai'(\omega^2z)\end{matrix}\right)e^{-\frac{\pi i}{6}\sigma_3}, & z \in \rm{I} \\
    \left(\begin{matrix} \Ai(z)&\quad \Ai(\omega^2z)\\ \Ai'(z)&\quad \omega^2\Ai'(\omega^2z)\end{matrix}\right)e^{-\frac{\pi i}{6}\sigma_3}\left(\begin{matrix} 1& 0\\-1&1\end{matrix}\right), & z \in \rm{II} \\
    \left(\begin{matrix} \Ai(z)&\quad -\omega^2\Ai(\omega z)\\ \Ai'(z)&\quad -\Ai'(\omega z)\end{matrix}\right)e^{-\frac{\pi i}{6}\sigma_3}\left(\begin{matrix} 1& 0\\1& 1\end{matrix}\right), & z \in \rm{III} \\
    \left(\begin{matrix} \Ai(z)&\quad -\omega^2\Ai(\omega z)\\ \Ai'(z)&\quad -\Ai'(\omega z)\end{matrix}\right)e^{-\frac{\pi i}{6}\sigma_3}, & z \in \rm{IV} \\
  \end{cases}
\end{equation}
with $\omega = e^{\frac{2}{3}\pi i }$; see \cite[eq.(3.62)]{van2007}. Here the regions I$-$IV are depicted in Fig. \ref{contour-airy}. For properties of Airy functions, one may refer to \cite{dlmf}.
\begin{figure}[h]
\begin{center}
\includegraphics[width=150pt]{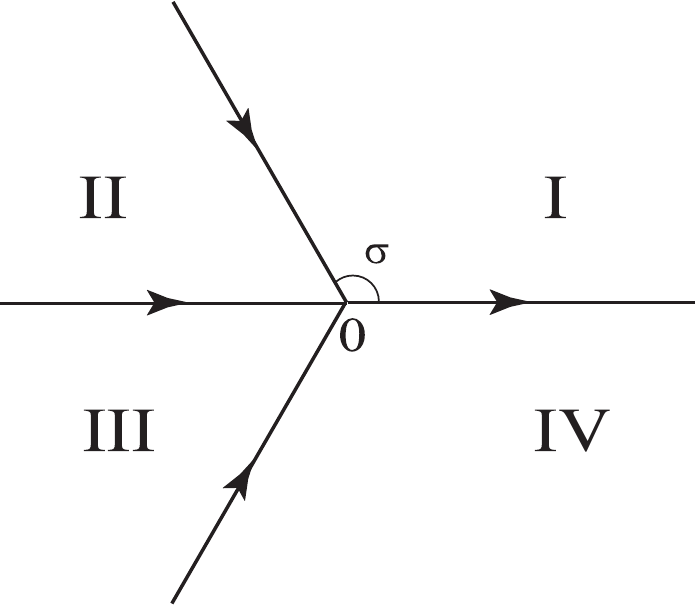}
\end{center}
\caption{Regions for the Airy parametrix with the angle $\sigma\in (\frac{\pi}{3}, \pi)$} \label{contour-airy}
\end{figure}

\subsection{Local parametrix near 0}

This parametrix is constructed in terms of Bessel functions as follows
\begin{equation} \label{T-0-bessel}
  T^{(0)}(z)=E^{(0)}(z)\Psi^{(0)}(\tilde{f}(z))e^{-N\xi(z)\sigma_3}(-z)^{-\frac{1}{2}\alpha\sigma_3},
\end{equation}
where
\begin{equation}
 E^{(0)}(z)=(-1)^N T^{(\infty)} (z) (-z)^{\frac{1}{2}\alpha
 \sigma_3}\frac{1}{\sqrt{2}}\left(\begin{matrix} 1&\quad i\\i&\quad
 1\end{matrix}\right)\tilde{f}(z)^{\frac{\sigma_3}{4}}(2\pi)^{\frac{\sigma_3}{2}},
\end{equation}
\begin{equation} \label{f-bessel-def}
 \tilde{f}(z):=\left(\frac{N }{2}\int_0^z
 \varphi(s)ds\right)^2 = \left(\frac{N}{4}\int_0^z
 \sqrt{\frac{s-\beta}{s}}h(s) ds\right)^2
\end{equation}
for $z\in \mathbb{C}\backslash[0,\infty)$, and
\begin{equation} %\label{bessel-para}
 \Psi^{(0)}(z)= \begin{cases}
    \left(\begin{matrix} I_\alpha(2z^{\frac{1}{2}})&  -\frac{i}{\pi} K_\alpha(2z^{\frac{1}{2}})\\-2\pi i z^{\frac{1}{2}}I_\alpha'(2z^{\frac{1}{2}})&  -2z^{\frac{1}{2}}K_\alpha'(2z^{\frac{1}{2}})\end{matrix}\right), & z \in \rm{I'} \\
    \left(\begin{matrix} \frac{1}{2}H_\alpha^{(1)}(2(-z)^{\frac{1}{2}})&  -\frac{1}{2} H_\alpha^{(2)}(2(-z)^{\frac{1}{2}})\\-\pi z^{\frac{1}{2}}(H_\alpha^{(1)})' (2(-z^{\frac{1}{2}}))&  \pi z^{\frac{1}{2}}(H_\alpha^{(2)})' (2(-z^{\frac{1}{2}}))\end{matrix}\right) e^{\frac{\alpha \pi i}{2} \sigma_3}, & z \in \rm{II'} \\
    \left(\begin{matrix} \frac{1}{2}H_\alpha^{(2)}(2(-z)^{\frac{1}{2}})&  \frac{1}{2} H_\alpha^{(1)}(2(-z)^{\frac{1}{2}})\\ \pi z^{\frac{1}{2}}(H_\alpha^{(2)})' (2(-z^{\frac{1}{2}}))&  \pi z^{\frac{1}{2}}(H_\alpha^{(1)})' (2(-z^{\frac{1}{2}}))\end{matrix}\right)e^{-\frac{\alpha \pi i}{2} \sigma_3}, & z \in \rm{III'};
  \end{cases}
\end{equation}
see \cite[eq.(3.81)]{van2007}. The regions I$'$$-$III$'$ are depicted in Fig. \ref{contour-bessel}. Here $I_\alpha$ and $K_\alpha$ are modified Bessel functions, and $H_\alpha^{(1)}$ and $H_\alpha^{(2)}$ are Hankel functions of the first and second kind, respectively. For properties of these functions, one may refer to \cite{dlmf} again.
\begin{figure}[h]
\begin{center}
\includegraphics[width=150pt]{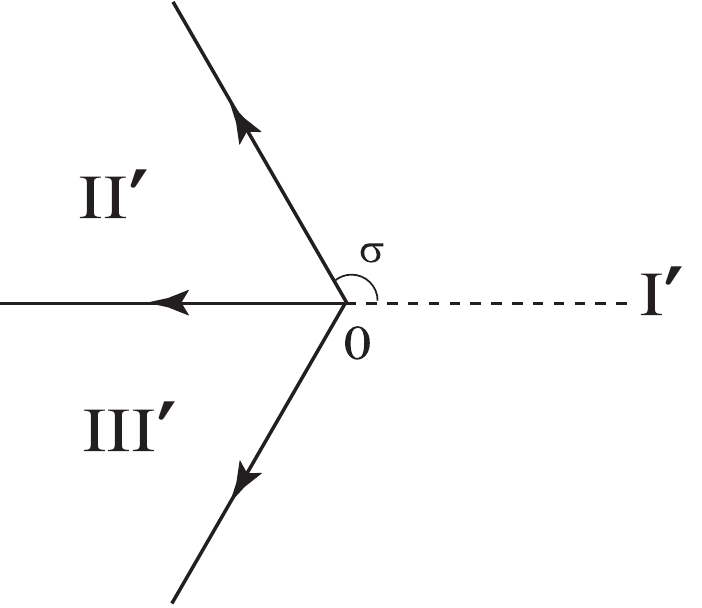}
\end{center}
\caption{Regions for the Bessel parametrix with the angle $\sigma\in (0, \pi)$} \label{contour-bessel}
\end{figure}

\subsection{The error}

Now we have the following approximation for $T(z)$
\begin{equation}
  T^{(A)}(z)= \begin{cases}
    T^{(\beta)}(z) & \textrm{for } z \in
U_\delta\backslash \Sigma_T, \\
T^{(0)}(z) &\textrm{for } z \in
\tilde{U}_\delta\backslash \Sigma_T,\\
T^{(\infty)}(z) &{\rm
elsewhere},
  \end{cases}
\end{equation}
where $U_\delta$ and $\tilde{U}_\delta$ are two disks with constant radius $\delta$, centered at $\beta$ and $0$, respectively. To see the difference between $T(z)$ and $T^{(A)}(z)$, we define
\begin{equation} \label{TtoS}
  S(z):=T(z) T^{(A)}(z)^{-1}.
\end{equation}
Then, $S(z)$ satisfies a RH problem as follows.

\begin{enumerate}

\item[($S_a$)] \quad $S(z)$ is analytic in $\mathbb{C}\backslash \Sigma_S$; see Fig. \ref{contour-s};

\item[($S_b$)]\quad $S_+(z)=S_-(z)J_S(z)$, where
\begin{equation} \label{S-rhp}
  J_S(z)= \begin{cases}
    T^{(0)}(z)T^{(\infty)}(z)^{-1} & \textrm{for } z \in \partial\tilde{U}_\delta\\
    T^{(\beta)}(z)T^{(\infty)}(z)^{-1} & \textrm{for } z \in \partial{U}_\delta\\
    T^{(\infty)}(z)J_T(z)T^{(\infty)}(z)^{-1} & \textrm{for } z \in \Sigma_{1,S}\bigcup \Sigma_{3,S}\bigcup \Sigma_{4,S};
  \end{cases}
\end{equation}

\item[($S_c$)]\quad$S(z)=I+O\left(\frac{1}{z}\right)$ as $z \to \infty$.

\end{enumerate}
\begin{figure}[h]
\begin{center}
\includegraphics[width=280pt]{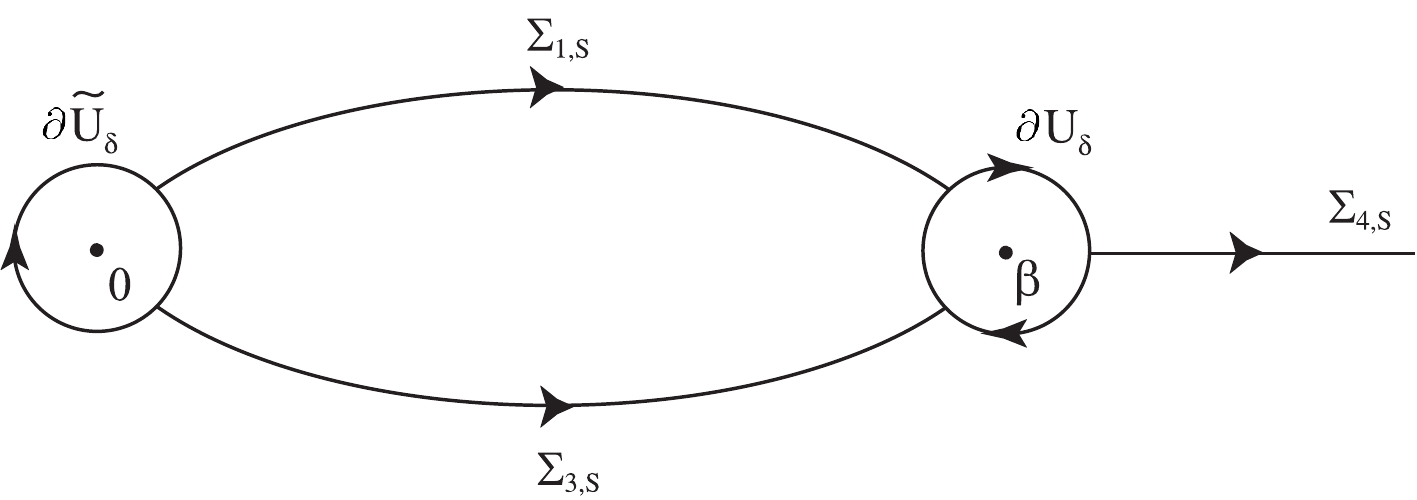}
\end{center}
\caption{Contour $\Sigma_S$} \label{contour-s}
\end{figure}

Based on the formulas for $J_T(z)$ and $T^{(\infty)}(z)$ in \eqref{jump-jt} and \eqref{out-para}, one can see $J_S(z)$ is exponentially close to the identity matrix for $z \in \Sigma_S\backslash (\partial \tilde{U}_\delta \cup \partial U_\delta)$ as $N \to \infty$. For $z \in \partial \tilde{U}_\delta \cup \partial U_\delta$, $J_S(z)$ possess the following asymptotic expansions as $N \to \infty$
\begin{equation} \label{jump-js}
  J_S(z)=\begin{cases}
    T^{(\beta)}(z) T^{(\infty)}(z)^{-1}\sim I + \D\sum_{k=1}^\infty
\frac{1}{N^k}J_k^{(\beta)}(z), & \textrm{for } z \in \partial U_\delta, \vspace{5pt} \\
T^{(0)}(z)T^{(\infty)}(z)^{-1} \sim I+\D\sum_{k=1}^{\infty}\frac{1}{N^k}J_k^{(0)}(z), & \textrm{for } z \in \partial \tilde{U}_\delta,
  \end{cases}
\end{equation}
where $J_k^{(\beta)}(z)$ and $J_k^{(0)}(z)$ are meromorphic functions in $U_\delta$ and $\tilde{U}_\delta$ with a pole at $\beta$ and 0, respectively. Moreover, they have explicit expressions as follows
\begin{equation} \label{jk-beta}
  \hspace{-10pt}J_k^{(\beta)}(z)=\frac{1}{2(\int_\beta^z \varphi(s)ds)^k}T^{(\infty)}(z) z^{\frac{\alpha}{2} \sigma_3} \left(\begin{matrix}
(-1)^k(s_k+t_k)& (s_k-t_k)i \\  (-1)^{k+1}(s_k-t_k)i & s_k+t_k\end{matrix}\right) z^{-\frac{\alpha}{2} \sigma_3} T^{(\infty)}(z)^{-1}
\end{equation}
and
\begin{align}
  J_k^{(0)}(z)& =\frac{(\alpha,k-1)(-1)^k}{2^k (\int_0^z \varphi(s)ds)^k}T^{(\infty)}(z) (-z)^{\frac{\alpha}{2}\sigma_3} \left(\begin{matrix}
\frac{(-1)^k}{k}\left(\alpha^2+ \frac{1}{2}k-\frac{1}{4}\right) &
\left(k-\frac{1}{2}\right)i \\ (-1)^{k+1}\left(k-\frac{1}{2}\right)i &
\frac{1}{k}\left(\alpha^2+ \frac{1}{2}k-\frac{1}{4}\right)\end{matrix}\right) \nonumber \\
& \quad \times (-z)^{-\frac{\alpha}{2}\sigma_3} T^{(\infty)}(z)^{-1}, \label{jk-0}
\end{align}
where
\begin{equation*}
  s_k=\frac{\Gamma(3k+\frac{1}{2})}{54^k k!\Gamma(k+\frac{1}{2})}, \quad  t_k=-\frac{6k+1}{6k-1}s_k, \quad (\alpha,k)=\frac{(4\alpha^2-1)\cdots[4\alpha^2-(2k-1)^2]}{2^{2k}k!}
\end{equation*}
for $ k \geq 1$; see \cite[eq.(3.75)\&(3.97)]{van2007}. Because the jump matrix tends to $I$ as $N \to \infty$, by the standard analysis in \cite{dkmvz2}, the RH problem for $S$ has a unique solution when $N$ is sufficiently large. Moreover, $S(z)$ satisfies the following asymptotic expansion
\begin{equation}
  S(z) \sim I + \sum_{k=1}^\infty N^{-k}S_k(z), \qquad \textrm{as } N \to \infty
\end{equation}
where $S_k(z)$ are bounded functions which are analytic in $\mathbb{C} \setminus ( \partial \tilde{U}_\delta \cup \partial U_\delta)$ and can be computed explicitly.

There is an important observation about explicit forms of $S_k(z)$. From formulas for $T^{(\infty)}(z)$, $J_k^{(\beta)}(z)$ and $J_k^{(0)}(z)$ in \eqref{out-para}, \eqref{jk-beta} and \eqref{jk-0}, one can see that they satisfy nice symmetric properties when $\alpha =0$. Due to these properties, we have the following results for $S_k(z)$, which are crucial in our subsequent analysis.
\begin{lem} \label{lemm-s-pattern}
When $\alpha=0,$ we have
\begin{equation} \label{sk-pattern}
  S_k(z) = \begin{cases}
    s_k^{(1)}(z)I + s_k^{(2)}(z)\sigma_2, & \textrm{when $k$ is even}, \\
    s_k^{(1)}(z)\sigma_3 + s_k^{(2)}(z)\sigma_1, & \textrm{when $k$ is odd},
  \end{cases}
\end{equation}
where $\sigma_1$ and $\sigma_2$ are the Pauli matrices
\begin{equation}
  \sigma_1 = \begin{pmatrix}
    0 & 1 \\ 1 & 0
  \end{pmatrix}, \qquad \sigma_2 = \begin{pmatrix}
    0 & -i \\ i & 0
  \end{pmatrix}.
\end{equation}
\end{lem}
\begin{proof}
  When $\alpha=0$, the outside parametrix $T^{(\infty)}(z)$ in \eqref{out-para} is simply given by
\begin{equation}
  T^{(\infty)}(z)= \frac{a(z)+a^{-1}(z)}{2} I + \frac{a(z)-a^{-1}(z)}{2} \sigma_2.
\end{equation}
Then, for $J_k^{(\beta)}(z)$ in \eqref{jk-beta}, we have
\begin{equation}
  \begin{aligned}
  J_k^{(\beta)}(z) & = \frac{2^{k-1}}{(\int_\beta^z \varphi(s)ds)^k}T^{(\infty)}(z) \biggl[ (s_k+t_k) I - (s_k-t_k)\sigma_2 \biggr]T^{(\infty)}(z)^{-1}\\&=\frac{2^{k-1}}{(\int_\beta^z \varphi(s)ds)^k} \biggl[ (s_k+t_k) I - (s_k-t_k)\sigma_2 \biggr], \qquad \textrm{when $k$ is even,} \end{aligned}
\end{equation}
and
\begin{equation}
  \begin{aligned}
  J_k^{(\beta)}(z)&=\frac{2^{k-1}}{(\int_\beta^z \varphi(s)ds)^k}  T^{(\infty)}(z) \biggl[ -(s_k+t_k) \sigma_3 + i(s_k-t_k)\sigma_1 \biggr] T^{(\infty)}(z)^{-1}\\
& = \frac{2^{k-1} }{(\int_\beta^z \varphi(s)ds)^k} \biggl[ -\biggl(a^2(z) t_k+ \frac{s_k}{a^{2}(z)}  \biggr) \sigma_3 - \left(a^2(z) t_k - \frac{s_k}{a^{2}(z)} \right) i \sigma_1 \biggr], \ \textrm{when $k$ is odd}.\end{aligned}
\end{equation}
Similarly, for $J_k^{(0)}(z)$ in \eqref{jk-0}, we get,
\begin{equation}
  J_k^{(0)}(z)=\frac{(0,k-1)(-1)^k}{2^k(\int_0^z \varphi(s)ds)^k} \biggl[ \left(\frac{1}{2} - \frac{1}{4k}\right) I - \left(k - \frac{1}{2}\right)\sigma_2 \biggr] , \qquad \textrm{when $k$ is even,}
\end{equation}
and
\begin{equation}\begin{aligned}
  J_k^{(0)}(z)=&\frac{(0,k-1)(-1)^k}{2^k(\int_0^z \varphi(s)ds)^k} \biggl[ \biggl(\frac{(2k-1)^2 a^2(z)}{8k}- \frac{(2k-1)(2k+1)}{8k\,a^{2}(z)}  \biggr) \sigma_3 \\
 & \qquad + \biggl(\frac{(2k-1)^2 a^2(z)}{8k}+ \frac{(2k-1)(2k+1)}{8k\,a^{2}(z)}  \biggr) i \sigma_1 \biggr], \ \textrm{when $k$ is odd.}\end{aligned}
\end{equation}
According to the above formulas, the jump matrix $J_S(z)$ in \eqref{jump-js} always satisfies an asymptotic expansion with symmetric properties. That is, its even terms are given as a linear combination of $I$ and $\sigma_2$, while its odd terms are expressed as a linear combination of $\sigma_1$ and $\sigma_3$. Therefore, using the same idea as in the proof of \cite[Lemma 3.5]{em2003}, it is not difficult to prove our results by mathematical induction.
This completes the proof of our lemma.
\end{proof}

\begin{rmk} \label{rmk-alpha}
  When $\alpha \neq 0$, it is still possible to calculate $J_k^{(\beta)}(z)$ and $J_k^{(0)}(z)$. But now the formulas are more complicated. For example, when $k$ is even, we have
  \begin{align*}
    J_k^{(\beta)}(z) & =  \frac{2^{k-1}(s_k+t_k)}{(\int_\beta^z \varphi(s)ds)^k}  I  + \frac{2^{k-1}(t_k-s_k)}{(\int_\beta^z \varphi(s) ds)^k} \biggl[(\frac{a+a^{-1}}{2})^2 \begin{pmatrix}
      0 & -i z^\alpha D^{-2} 4^{-\alpha} \\ i z^{-\alpha} D^{2} 4^{\alpha} & 0
    \end{pmatrix} \\
    & - (\frac{a-a^{-1}}{2})^2 \begin{pmatrix}
      0 & -i z^{-\alpha} D^{2} 4^{-\alpha} \\ i z^{\alpha} D^{-2} 4^{\alpha} & 0
    \end{pmatrix} + \frac{a^2-a^{-2}}{4} \begin{pmatrix}
      z^{-\alpha} D^{2} &  0 \\ 0 & z^{\alpha} D^{-2}
    \end{pmatrix} \\
    & - \frac{a^2-a^{-2}}{4} \begin{pmatrix}
      z^{\alpha} D^{-2} & 0 \\ 0 & z^{-\alpha} D^{2}
    \end{pmatrix}  \biggr],
  \end{align*}
  where $a$ and $D$ denote the functions $a(z)$ and $D(z)$ in \eqref{fun-a(z)} and \eqref{fun-d(z)} for brevity. For such kind of $J_k^{(\beta)}(z)$ and $J_k^{(0)}(z)$, the pattern for $S_k(z)$ in \eqref{sk-pattern} is no longer valid.  Therefore, it is likely that one can only obtain the asymptotic expansion in powers of $1/N$ instead of $1/N^2$ in Theorem \ref{thm2}. This is also the reason why we focus on the case $\alpha=0$ in this paper.
\end{rmk}

%%%%%%%%%%%%%%%%%%%%%%%%%%%%%%%%%%%%%%%%%%%%%%%%%%%%%%%%%%%%%%%%%%%%%%%%%%%%%%%%%%%%%%%%%%%%%%%%

\section{Asymptotic expansion of $\rho_N^{(1)} (z)$} \label{sec-rho&asy}

With the Deift-Zhou steepest descent analysis done in the previous section, we are ready to obtain the asymptotic expansion for $\rho_N^{(1)} (z)$. Reverting the transformations in \eqref{YtoU}, \eqref{UtoT} and \eqref{TtoS}, we have
\begin{equation} \label{Y-ST}
  \begin{aligned}
Y(z)& = e^{\frac{N l_V}{2}\sigma_3} U(z) e^{N(g(z)-\frac{l_V}{2})\sigma_3}\\
 & = e^{\frac{N l_V}{2}\sigma_3} T(z)
\left(\begin{matrix} 1 &\quad 0\\r(z)&\quad 1\end{matrix}\right) e^{N(g(z)-\frac{l_V}{2})\sigma_3}\\
& = e^{\frac{N l_V}{2}\sigma_3} S(z)T^{(A)}(z)\left(\begin{matrix} 1 &\quad
0\\r(z)&\quad 1\end{matrix}\right) e^{N(g(z)-\frac{l_V}{2})\sigma_3},
\end{aligned}
\end{equation}
where
$$
r(z)= \begin{cases}
  0, & \text{for $z$ outside the lens-shaped region in Fig. \ref{fig-gamma}} \\
  \pm e^{-2N\xi(z)}, & \text{for $z$ in the upper/lower lens region in Fig. \ref{fig-gamma}.}
\end{cases}
$$
Recalling \eqref{rho-pn} and \eqref{Y-sol}, we get
\begin{equation}
  \rho_N^{(1)} (z)=-\frac{e^{-N V_\textbf{t}(z)}}{2\pi i N}[Y_{11}'(z)Y_{21}(z)-Y_{11}(z)Y_{21}'(z)].
\end{equation}
For $z$ in the neighborhood of $\beta$, since $Y(z)$ involves the same Airy-type parametrix, we have the same formula as that in \cite[eq.(4.4)]{em2003}. For $z$ in the neighborhood of 0, as the Bessel-type paramatrix is constructed, the computations are different. First, we get the following expression.

\begin{lem} \label{lemma-rho}
  Let $\textbf{t}$ belong to the set $\mathbb{T}(\mathcal{T},\gamma)$ in Theorem \ref{thm-measure}. With $a(z)$ and $\tilde{f}(z)$ defined in \eqref{fun-a(z)} and \eqref{f-bessel-def}, we have
  \begin{equation} \label{rho-for-1}
    \begin{aligned}
    \rho_N^{(1)}(z) &=\frac{2}{N} \left(\frac{a'}{a}+\frac{\tilde{f}'}{4\tilde{f}}\right)\tilde{f}^{1/2}I_0(2\tilde{f}^{1/2}) I_0'(2\tilde{f}^{1/2}) - \frac{\tilde{f}'}{N}\left[I_0^2(2\tilde{f}^{1/2})-(I_0'(2\tilde{f}^{1/2}))^2\right] \\
    & \hspace{-25pt} -\frac{1}{4N\pi i} \left[(S'(z)B^{(0)}(z)\tilde\Psi(\tilde{f}))_{11}(S(z)B^{(0)}(z)\tilde\Psi(\tilde{f}))_{21}-(S(z)B^{(0)}(z)\tilde\Psi(\tilde{f}))_{11}
    (S'(z)B^{(0)}(z)\tilde\Psi(\tilde{f}))_{21}\right]
    \end{aligned}
  \end{equation}
  for $z$ in a fixed-size complex neighborhood of 0. Here $S(z)$ satisfies the RH problem in \eqref{S-rhp},
  \begin{equation}
    \tilde{\Psi}(z):=\left(\begin{matrix} I_0(2z^{1/2}) &\quad
    -\frac{i}{\pi}K_0(2z^{1/2})\\-2\pi i z ^{1/2} I_0'(2z^{1/2})&\quad
    -2z^{1/2}K_0'(2z^{1/2})\end{matrix}\right), \qquad \textrm{for } z \in \mathbb{C} \setminus (-\infty, 0],
  \end{equation}
  and
  \begin{equation} \label{fun-b0}
    B^{(0)}(z)=\left(\begin{matrix} a(z) & ia^{-1}(z) \\ ia(z) & a^{-1}(z) \end{matrix}\right)\tilde{f}(z)^{\frac{\sigma_3}{4}}(2\pi)^{\frac{\sigma_3}{2}}.
  \end{equation}
\end{lem}
\begin{proof}
  According to the parametrix in \eqref{T-0-bessel}, we have from \eqref{g and v} and \eqref{Y-ST}
  \begin{equation} \label{Y-SE}
    Y(z) =  e^{\frac{N l_V}{2}\sigma_3} S(z)E^{(0)}(z) \tilde{\Psi}(\tilde{f}(z))e^{\frac{N}{2} V_\textbf{t}(z) \sigma_3},
  \end{equation}
  where
  \begin{equation*}
    \begin{aligned}
    E^{(0)}(z)=&(-1)^N T^{(\infty)}(z)\frac{1}{\sqrt{2}} \left(\begin{matrix} 1 &\quad i\\i&\quad
    1\end{matrix}\right) \tilde{f}(z)^{\frac{\sigma_3}{4}}(2\pi)^{\frac{\sigma_3}{2}}\\
    =&\frac{(-1)^N}{\sqrt {2}} \left(\begin{matrix} a &\quad ia^{-1}\\ia&\quad a^{-1}\end{matrix}\right)\tilde{f}(z)^{\frac{\sigma_3}{4}}(2\pi)^{\frac{\sigma_3}{2}} =\frac{(-1)^N}{\sqrt {2}}B^{(0)}(z).
    \end{aligned}
  \end{equation*}
  As $Y_{11}'(z)$ and $Y_{21}'(z)$ appear in the formula for $\rho_N^{(1)} (z)$, we need to calculate the derivative for \eqref{Y-SE}. From the formula for $E^{(0)}(z)$ in the above formula, we have
  \begin{eqnarray}
    \frac{d}{dz} E^{(0)}(z)&=& \frac{(-1)^N}{\sqrt {2}}  \left[\left(\begin{matrix} a' &  -ia^{-2}a'\\ia'&
    -a^{-2}a'\end{matrix}\right) \tilde{f}^{\frac{\sigma_3}{4}} + \left(\begin{matrix} a &
    ia^{-1}\\ia&  a^{-1}\end{matrix}\right) \left(\begin{matrix} \frac{1}{4}\tilde{f}'\tilde{f}^{-\frac{3}{4}} &  0\\0&
    -\frac{1}{4}\tilde{f}'\tilde{f}^{-\frac{5}{4}}\end{matrix}\right) \right] (2\pi)^{\frac{\sigma_3}{2}} \nonumber \\
    &=& \left[\frac{a'}{a}+\frac{\tilde{f}'}{4\tilde{f}}\right]E^{(0)}(z)\sigma_3.
  \end{eqnarray}
  To calculate $\frac{d}{d z}\tilde{\Psi}(z)$, we recall that both $I_\alpha (z)$ and $K_\alpha(z)$ satisfy the modified Bessel equation
    $$
    z^2 u''(z) + z u'(z)-(z^2+\alpha^2)u(z)=0.
    $$
  Then we get
  \begin{equation} \label{Psi-dev}
    \frac{d}{dz}\tilde{\Psi}(z)= \left(\begin{matrix}
    z^{-1/2}I_0'(2z^{1/2}) &\quad
    -\frac{i}{\pi}z^{-1/2}K_0'(2z^{1/2})\\-2\pi i I_0(2z^{1/2})&\quad -2
    K_0(2z^{1/2})\end{matrix}\right)= \left(\begin{matrix}  0&\quad -\frac{1}{2\pi iz}\\-2\pi i &\quad
    0\end{matrix}\right)\tilde{\Psi}(z).
  \end{equation}
  Combining the above formulas \eqref{Y-SE}-\eqref{Psi-dev}, we obtain \eqref{rho-for-1}.
\end{proof}

To derive the full asymptotic expansion for $\rho_N^{(1)}(z)$, some more detailed computations about the last term in \eqref{rho-for-1} are needed. Based on the pattern for the asymptotic expansion of $S(z)$ in Lemma \ref{lemm-s-pattern}, we get the following result.

\begin{lem} \label{lemma-rho-exp}
    We have
    \begin{equation} \label{rho-for-2}
      \begin{aligned}
    &-\frac{1}{4N\pi i}\left[(S'B^{(0)}\tilde{\Psi}(\tilde{f}))_{11}(SB^{(0)}\tilde{\Psi}(\tilde{f}))_{21}
    -(SB^{(0)}\tilde{\Psi}(\tilde{f}))_{11}(S'B^{(0)}\tilde{\Psi}(\tilde{f}))_{21}\right]\\
    & \qquad \sim \sum_{j \textrm{ even, } j \ge 2}  N^{-j}\tilde{a}_j(z)a^2 \tilde{f}^{1/2}I_0^2(2\tilde f^{1/2}) +\sum_{j \textrm{ even, } j\ge 2} N^{-j}\tilde{b}_j(z)  \frac{\tilde f^{1/2} I_0'(2\tilde f^{1/2})^2}{a^2 } \\
    & \qquad \qquad +\sum_{j \textrm{ odd, } j \ge 3}  N^{-j}\tilde{c}_j(z) \tilde f^{1/2}I_0(2\tilde f^{1/2})I_0'(2\tilde f^{1/2}), \qquad \textrm{as } N \to \infty,
    \end{aligned}
    \end{equation}
    where $\tilde{a}_j(z)$, $\tilde{b}_j(z)$ and $\tilde{c}_j(z)$ are analytic functions in both $z$ and $\textbf{t}$, for $z$ in a fixed-size neighborhood of 0 and $\textbf{t}$ belong to the set $\mathbb{T}(\mathcal{T},\gamma)$ in Theorem \ref{thm-measure}.
\end{lem}

\begin{proof} Recalling Lemma \ref{lemm-s-pattern}, we have
\begin{equation*}
S(z) \sim I+\sum_{k \textrm{ odd, } k\ge
1}(s_k^{(1)}\sigma_3+s_k^{(2)}\sigma_1)N^{-k}+\sum_{k \textrm{ even, }
\hbox{ }k\ge 2}(s_k^{(1)}I+s_k^{(2)}\sigma_2)N^{-k}, \quad \textrm{as } N \to \infty
\end{equation*}
which gives us
\begin{equation*}
S'(z) \sim \sum_{k \textrm{ odd, }k \ge
1}(s_k^{(1)'}\sigma_3+s_k^{(2)'}\sigma_1)N^{-k} +\sum_{k \textrm{ even, } k \ge 2}(s_k^{(1)'}I+s_k^{(2)'}\sigma_2)N^{-k}, \quad \textrm{as } N \to \infty.
\end{equation*}
Consider the first column of
\begin{equation}
B^{(0)}(z) \tilde\Psi(\tilde{f}) =\left(\begin{matrix} m_{11} &\quad *
\\m_{21}&\quad *\end{matrix}\right),
\end{equation}
where $B^{(0)}(z)$ is given in \eqref{fun-b0}. It is easily verified that
\begin{equation} \label{m&m}
  \begin{aligned}
    &m_{11}^2+m_{21}^2= 8 \pi \tilde f^{1/2}I_0(2\tilde f^{1/2})I_0'(2\tilde f^{1/2}), \\
    &m_{11}^2-m_{21}^2= 4 \pi \tilde{f}^{1/2} \left( a^2 I_0^2(2\tilde f^{1/2}) - \frac{I_0'(2\tilde f^{1/2})^2}{a^2 } \right), \\
    &m_{11} m_{21} = 2 \pi i \tilde{f}^{1/2} \left( a^2 I_0^2(2\tilde f^{1/2}) - \frac{I_0'(2\tilde f^{1/2})^2}{a^2 } \right).
  \end{aligned}
\end{equation}
From the above formulas, we have
\begin{equation}
\begin{aligned}&(S'B^{(0)}\tilde{\Psi})_{11}(SB^{(0)}\tilde{\Psi})_{21} \sim \left[\sum_k \frac{s_{2k-1}'m_{11}+t_{2k-1}'m_{21}}{N^{2k-1}}+\sum_k \frac{s_{2k}'m_{11} - i \, t_{2k}'m_{21}}{N^{2k}}\right]\\
& \qquad \qquad \times \left[m_{21}+\sum_k \frac{t_{2k-1}m_{11}-s_{2k-1}m_{21}}{N^{2k-1}}+\sum_k \frac{i\,t_{2k}m_{11}+s_{2k}m_{21}}{N^{2k}}\right]\end{aligned}
\end{equation}
and
\begin{eqnarray}
  &&(SB^{(0)}\tilde{\Psi})_{11}(S'B^{(0)}\tilde{\Psi})_{21} \sim \left[m_{11}+\sum_k \frac{s_{2k-1}m_{11}+t_{2k-1}m_{21}}{N^{2k-1}} + \sum_k \frac{s_{2k}m_{11} - i\,t_{2k}m_{21}}{N^{2k}}\right] \nonumber \\
&& \qquad \qquad \times \left[\sum_k \frac{t_{2k-1}'m_{11}-s_{2k-1}'m_{21}}{N^{2k-1}}+\sum_k \frac{i \, t_{2k}'m_{11}+s_{2k}'m_{21}}{N^{2k}}\right].
\end{eqnarray}
Substracting the above two formulas, one can see that all odd terms of $1/N$ are given in terms
of $m_{11}^2-m_{21}^2$ and $m_{11}m_{21}$. On the other hand, all even terms of $1/N$ are given in terms of $m_{11}^2+m_{21}^2$. Recalling the properties of $m_{11}$ and $m_{12}$ in \eqref{m&m}, one gets \eqref{rho-for-2}. And analyticity of the coefficients $\tilde{a}_j(z)$, $\tilde{b}_j(z)$ and $\tilde{c}_j(z)$ follows from the analytic property of $S(z)$.
\end{proof}

%%%%%%%%%%%%%%%%%%%%%%%%%%%%%%%%%%%%%%%%%%%%%%%%%%%%%%%%%%%%%%%%%%%%%%%%

\section{Proof of the main theorem} \label{sec-mainthm}

Now we are ready to prove Theorem \ref{thm2} based on
the representations of the one-point correlation function $\rho_N^{(1)}(z)$
obtained in the previous section. We would like to carry on our asymptotic
analysis in the neighborhood of 0 and $\beta$ separately. To achieve
this object, let us introduce a partition $\{\chi_0,\chi_\beta\}$ of
unity for $(0,\infty)$ as follows
\begin{align*}
&\chi_0(z) \textrm{ is } C^\infty \quad \textrm{with} \quad 0 \leq \chi_0(z) \leq 1\\
&\overline{\textrm{supp } \chi_0} \subset (0, z^*+\varepsilon), \quad \chi_0 \equiv 1, \quad \textrm{for} \quad z \in [0, z^*-\varepsilon) \\
& \textrm{and } \chi_\beta(z)=1-\chi_0(z),
\end{align*}
where $z^*$ is a fixed point in  $(0, \beta)$.
Note that $\rho_N^{(1)}(z)$ is exponentially small for large $N$ when $z$ is bounded away from $[0,\beta]$. Then, instead of considering $\D\int_0^\infty \Theta(\lambda)\rho_N^{(1)}(\lambda)d\lambda$, it is enough to study $\D\int_0^{\beta+\delta} \Theta(\lambda)\rho_N^{(1)}(\lambda)d\lambda$ for some small $N$-independent $\delta$. Recall that $\Theta(\lambda)$ is a $C^\infty$-smooth function and grows no faster than a polynomial for $\lambda \to \infty$. According to the partition of unity above, we rewrite $\rho_N^{(1)}(z)$ as
\begin{equation}
\rho_N^{(1)}(z)=\chi_0(z)\rho_N^{(1,0)}(z)+\chi_\beta(z)\rho_N^{(1,\beta)}(z),
\end{equation}
and obtain
\begin{equation}\label{rho-exp-full}
\int_0^{\beta+\delta} \Theta(z)\rho_N^{(1)}(z)dz=\int_0^{z^*+\varepsilon} \chi_0(z)\Theta(z)\rho_N^{(1,0)}(z)dz+\int_{z^*-\varepsilon}^{\beta+\delta}\chi_\beta(z) \Theta(z)\rho_N^{(1,\beta)}(z)dz.
\end{equation}
Because $\rho_N^{(1,\beta)}(z)$ satisfies the same Airy-type representation as that in \cite[eq.(4.4)]{em2003}, the second term on the right-hand side of the above formula has an asymptotic expansion in powers of $1/N^2$ according to the results in \cite{em2003}.  Then it is sufficient for us to consider only the first term
\begin{equation}\label{rho-exp-1}
\int_0^{z^*+\varepsilon} \chi_0(z)\Theta(z)\rho_N^{(1,0)}(z) d z.
\end{equation}
According to the asymptotic expansion of $\rho_N^{(1)}$ in Lemma \ref{lemma-rho} and \ref{lemma-rho-exp}, there are four types of integrals in \eqref{rho-exp-1} as follows
\begin{itemize}
\item[(i)]  $\frac{1}{N}\int_0^{z^*+\varepsilon} \theta(z) \left[I_0^2(2\tilde f^{1/2})-(I_0'(2\tilde f^{1/2}))^2\right] \tilde f' dz$;

\item[(ii)]  ${N}^{-j+1}\int_0^{z^*+\varepsilon}\theta(z) I_0(2\tilde f^{1/2})I_0'(2\tilde f^{1/2}) dz\quad(j \hbox{ } odd)$;

\item[(iii)]  ${N}^{-j+1}\int_0^{z^*+\varepsilon}\theta(z) I_0^2(2\tilde f^{1/2}) dz\quad(j \hbox{ } even)$;

\item[(iv)]  ${N}^{-j+1}\int_0^{z^*+\varepsilon}\theta(z) (I_0'(2\tilde f^{1/2}))^2 dz\quad(j \hbox{ } even)$;
\end{itemize}
where $\theta(z)$ is a general infinitely differentiable function of $z$ and is compactly supported within $(0, z^*+\varepsilon)$.

Next, we will show that all integrals of these four types
possess asymptotic expansions in powers of $1/N^2$. Before going to the proof of their expansions, we would like to mention the following relations among these four types integrals.

\begin{prop} \label{prop-four-int}
If integrals of the types (i) have an
asymptotic expansion in powers of $1/N^2$, then integrals of all other three types have asymptotic expansions in powers of $1/N^2$.
\end{prop}
\begin{proof}
  First, for an even integer $j$, we have the following relations between type (i) and (iv) from integration by parts
\begin{eqnarray*}
&&{N}^{-j+1}\int_0^{z^*+\varepsilon}\theta(z) (I_0'(2\tilde f^{1/2}))^2 dz ={N}^{-j+1}\int_0^{z^*+\varepsilon}\left(\frac{\theta(z)\tilde f}{\tilde f'}\right) (I_0'(2\tilde f^{1/2}))^2 \tilde f^{-1}\tilde f'dz\notag\\
&& \hspace{2.2cm} =-{N}^{-j+1}\int_0^{z^*+\varepsilon}\left(\frac{\theta(z)\tilde f}{\tilde f'}\right)' \left[I_0^2(2\tilde f^{1/2})-(I_0'(2\tilde f^{1/2}))^2\right]dz\notag\\
&&\hspace{2.2cm} =-{N}^{-j-1}\int_0^{z^*+\varepsilon}\left(\frac{\theta(z)\tau}{\tau'}\right)'\frac{1}{\tau'}
\left[I_0^2(2\tilde f^{1/2})-(I_0'(2\tilde f^{1/2}))^2\right]\tilde f'dz.
\end{eqnarray*}
Note that, in the integration by parts, the boundary terms don't
appear because $\theta(z)$ is compactly supported within $(0,
z^*+\varepsilon)$. Moreover, $\tilde f'(z)$ is analytic and
nonvanishing  throughout the region of integration so that the
integral is well defined. Lastly,
$$
\tau'(z):=N^{-2} \tilde f'(z)
$$
defines an analytic nonvanishing function which is
independent of $N$; see the definition of $\tilde f(z)$ in \eqref{f-bessel-def}. As the quantity
$\left({\theta(z)\tau}/{\tau'}\right)'({1}/{\tau'})$ is
infinitely differentiable on $[0, z^*+\varepsilon]$, this suggests
type (iv) integral can be reduced to type (i).

Similarly, for the type (iii), we have
\begin{align*}
&{N}^{-j+1}\int_0^{z^*+\varepsilon}\theta(z) I_0^2(2\tilde f^{1/2}) dz \notag\\
=&{N}^{-j+1}\int_0^{z^*+\varepsilon}\frac{\theta(z)}{\tilde f'} \left[I_0^2(2\tilde f^{1/2})-(I_0'(2\tilde f^{1/2}))^2\right]\tilde f'dz +{N}^{-j+1}\int_0^{z^*+\varepsilon}\theta(z)(I_0'(2\tilde f^{1/2}))^2 dz\notag\\
=&{N}^{-j-1}\int_0^{z^*+\varepsilon}\frac{\theta(z)}{\tau'}
\left[I_0^2(2\tilde f^{1/2})-(I_0'(2\tilde f^{1/2}))^2\right]\tilde
f'dz +{N}^{-j+1}\int_0^{z^*+\varepsilon}\theta(z)(I_0'(2\tilde
f^{1/2}))^2 dz.
\end{align*}
The two terms in the last equation are integrals of the type (i) and
(iv), respectively.

Finally, suppose that $j$ is an odd integer, then, for the type
(ii), we have
\begin{align*}\label{e58}
&2{N}^{-j+1}\int_0^{z^*+\varepsilon}\theta(z) I_0(2\tilde f^{1/2})I_0'(2\tilde f^{1/2}) dz =2{N}^{-j+1}\int_0^{z^*+\varepsilon} \frac{\theta(z)\tilde f^{1/2}}{\tilde f'}I_0(2\tilde f^{1/2})I_0'(2\tilde f^{1/2})\tilde f^{-1/2}\tilde f' dz\notag \\
&\qquad =-{N}^{-j+1}\int_0^{z^*+\varepsilon} \left[\frac{\theta(z)\tilde f^{1/2}}{\tilde f'}\right]'I_0^2(2\tilde f^{1/2}) dz= -{N}^{-j}\int_0^{z^*+\varepsilon}
\left[\frac{\theta(z)\tau^{1/2}}{\tau'}\right]'I_0^2(2\tilde f^{1/2}) dz,
\end{align*}
where the last integral is of type (iii). Thus, there four types of integrals are equivalent and our results follow.
\end{proof}

Now we focus on the type (i) integral and derive its asymptotic expansion. We have
\begin{lem} \label{main-lemma}
  The type (i) integral
  \begin{equation}\label{e510}
\frac{1}{N}\int_0^{z^*+\varepsilon}\theta(z) \left[I_0^2(2\tilde
f^{1/2})-(I_0'(2\tilde f^{1/2}))^2\right] \tilde f' dz
\end{equation}
has an asymptotic expansion in powers of $1/N^2$.
\end{lem}
We will apply integration by parts and prove the above Lemma by induction. To achieve this goal, we need a few preliminary results. Define
\begin{equation}\label{zeta-def}
\zeta ^{3/2}(z)=\frac{N}{2}\int_0^z \sqrt{\frac{\beta-s}{s}}h(s)ds,
\end{equation}
then, we have from the definition of $\tilde f$ in \eqref{f-bessel-def}
\begin{equation} 
4\tilde f=e^{\pm\pi i}\zeta^3=-\zeta^3 \quad \textrm{and} \quad
4\tilde f'=-3\zeta^2\zeta'.
\end{equation}
With the help of the above formula, (\ref{e510}) can be rewritten as
\begin{align}\label{e514}
&\frac{1}{N}\int_0^{z^*+\varepsilon}\theta(z) \left[I_0^2(2\tilde
f^{1/2})-(I_0'(2\tilde f^{1/2}))^2\right] \tilde f' dz =-\frac{3}{4N}\int_0^{z^*+\varepsilon}\theta(z) F_0(\zeta) \zeta' dz,
\end{align}
where $F_0(\zeta)$ is an analytic function in $\zeta$ defined below
\begin{equation}\label{e515}
F_0(\zeta):=\left[J_0^2(\zeta^{3/2})+(J_0'(\zeta^{3/2}))^2\right]
\zeta^2.
\end{equation}

\begin{prop}
  The function $F_0$ satisfies the following asymptotic expansion as $\zeta \to \infty$.
\begin{equation} 
F_0(\zeta)\sim\zeta^{1/2}\left(\sum_{i=0}^\infty c_i \zeta^{-3i}+
\sum_{i=1}^\infty\bar c_i
\zeta^{-3i}\sin(2\zeta^{3/2})+\sum_{i=1}^\infty\tilde c_i
\zeta^{-3i+7/3}\cos (2\zeta^{3/2})\right),
\end{equation}
where $c_i$, $\bar c_i$ and $\tilde c_i$ are constants which can be
computed explicitly.
\end{prop}
\begin{proof}
  It is well-known that the
Bessel functions $J_\nu(z)$ and its derivative satisfy the following
asymptotic expansion as $z \to \infty$
\begin{equation}\label{e551}
J_0(z)\sim\left(\frac{2}{\pi
z}\right)^{1/2}\left(\cos\left(z-\frac{\pi}{4}\right) w_1-\sin
\left(z-\frac{\pi}{4}\right) w_2\right), \quad\quad |\arg
z|<\pi-\delta,
\end{equation}
\begin{equation} \label{e552}
J_0'(z) \sim-\left(\frac{2}{\pi
z}\right)^{1/2}\left(\sin\left(z-\frac{\pi}{4}\right) w_3+\cos
\left(z-\frac{\pi}{4}\right) w_4\right), \quad\quad |\arg
z|<\pi-\delta,
\end{equation}
where
\begin{equation}\label{e555}
w_1\sim\sum_{k=0}^\infty (-1)^k \frac{a_{2k}}{z^{2k}}, \quad
w_2\sim\sum_{k=0}^\infty (-1)^k \frac{a_{2k+1}}{z^{2k+1}}, \quad
w_3\sim\sum_{k=0}^\infty (-1)^k \frac{b_{2k}}{z^{2k}}, \quad
w_4\sim\sum_{k=0}^\infty (-1)^k \frac{b_{2k+1}}{z^{2k+1}};
\end{equation}
see \cite[\S 10.17]{dlmf}.
The coefficients $a_k$ and $b_k$ are known explicitly
\begin{equation*}
a_k=\frac{(4-1^2)(4-3^2)\cdots(4-(2k-1)^2)}{k!8^k}, \quad\quad b_k=\frac{(-1)^{k-1}((2k-3)!!)^2(4k^2-1)}{k!8^k},
\end{equation*}
with $a_0=1$, $b_0=1$ and $b_1=3/8$.

Substitute the expansions of the Bessel functions
(\ref{e551}) and (\ref{e552}) into (\ref{e515}), we get the expansion for $F_0(\zeta)$ as follows
\begin{equation}
F_0(\zeta)\sim \frac{1}{\pi}\zeta^{1/2}(f_0^{(0)}+f_S^{(0)}\sin(2\zeta^3/2)+f_C^{(0)}\cos(2\zeta^3/2)),
\end{equation}
where $f_0^{(0)}$, $f_S^{(0)}$ and $f_C^{(0)}$ are asymptotic series
valid for $\zeta \to \infty$,
\begin{align*}
f_0^{(0)}=w_1^2+w_2^2+w_3^2+w_4^2, \quad
f_S^{(0)}=w_1^2-w_2^2-w_3^2+w_4^2, \quad
f_C^{(0)}=2(w_1w_2-w_3w_4).
\end{align*}
Now using definitions (\ref{e555}) of the asymptotic series
$\{w_j\}_{j=1}^4$, we see that $f_0^{(0)}$ is an asymptotic
expansion in powers of $\zeta^{-3}$ starting with the constant term.
Similarly, $f_S^{(0)}$ is an asymptotic expansion in powers of
$\zeta^{-3}$ starting with $\zeta^{-3}$, and $f_C^{(0)}$ is an
asymptotic expansion in powers of the form $\zeta^{-3j-3/2}$,
$j=0,1,\cdots$. This finishes the proof of our proposition.
\end{proof}
We need one more lemma as follows.

\begin{lem} \label{lemma-exp}
If a function $F(\zeta)$ is $C^{\infty}$ on $[0, \infty)$, and
possess the following asymptotic expansion
\begin{equation}
F(\zeta) \sim \sum_{i=0}^\infty c_i\zeta^{-j/2-3i/2} \ \rm{trig}
(2\zeta^{3/2}), \qquad \textrm{as } \zeta \to \infty,
\end{equation}
in which $j \in \mathbb{N}$, and $\rm{trig}(\cdot)$ denotes
either $\sin(\cdot)$ or $\cos(\cdot)$, then the function
$\int_s^\infty F(\zeta)d\zeta$ possesses the following asymptotic
expansion as $s \to \infty$:
\begin{equation}
\int_s^\infty F(\zeta)d\zeta \sim \sum_{i=0}^\infty  c_i^*
s^{-(j+1)/2-3i/2}\ \rm{trig}(2\zeta^{3/2}).
\end{equation}
\end{lem}
\begin{proof}
The proof is similar to that of \cite[Lemma 5.5]{em2003}.
\end{proof}

Now let us derive the first two terms of the asymptotic expansion of \eqref{e510}, together with its error term.
\begin{prop}
  We have
  \begin{equation}
-\frac{3}{4N}\int_0^{z^*+\varepsilon}\theta(z) F_0(\zeta) \zeta' dz=\hat e_0+\frac{1}{N^2}\hat e_1+A_2,
\end{equation}
where $\hat e_0$ and $\hat e_1$ are independent of $N$ and $|A_2| \leq CN^{-7/3}$ for a constant $C$.
\end{prop}
\begin{proof}
  We rewrite the right-hand side of  (\ref{e514}) as
\begin{equation}
-\frac{3}{4N}\int_0^{z^*+\varepsilon}\theta(z) F_0(\zeta) \zeta' dz=\hat e_0+A_1,
\end{equation}
where
\begin{equation}
\hat e_0=-\frac{3}{4N}\int_0^{z^*+\varepsilon}c_0\theta(z)\zeta^{1/2}
\zeta' dz,
\end{equation}
and
\begin{equation}\label{e525}
A_1=-\frac{3}{4N}\int_0^{z^*+\varepsilon}\theta(z)
\left(F_0(\zeta)-c_0\zeta^{1/2}\right) \zeta' dz.
\end{equation}
Although there is a factor $N$ in the leading term $\hat e_0$, it is actually $N$-independent
\begin{align}
\hat e_0=&-\frac{3}{4N}\int_0^{z^*+\varepsilon}c_0\theta(z)\zeta^{1/2} \zeta' dz=-\frac{1}{2N}\int_0^{z^*+\varepsilon}c_0\theta(z)d\zeta^{3/2}\notag\\
=& -\frac{c_0}{4}\int_0^{z^*+\varepsilon}\theta(z)\sqrt{\frac{\beta-z}{z}}h(z)dz;
\end{align}
see the definition of $\zeta(z)$ in \eqref{zeta-def}.
To get the next term, we apply integration by parts in (\ref{e525}) and obtain
\begin{equation}
-\frac{3}{4N}\int_0^{z^*+\varepsilon}\theta(z) F_0(\zeta) \zeta' dz=\hat e_0+\frac{3}{4N}\int_0^{z^*+\varepsilon}\theta'(z) F_1(\zeta) dz
\end{equation}
with
\begin{equation}\label{e517}
F_1'(\zeta)=F_0(\zeta)-c_0\zeta^{1/2}.
\end{equation}
We don't have any contributions from the boundary terms because $\theta(z)$ is compactly
supported within $(0, z^*+\varepsilon)$. Next, we repeat integration by parts twice to produce the higher order terms
\begin{equation}
-\frac{3}{4N}\int_0^{z^*+\varepsilon}\theta(z) F_0(\zeta) \zeta' dz=\hat e_0+\hat e_1 N^{-2}+A_2,
\end{equation}
where
\begin{equation}\label{e530}
\hat e_1=\frac{3}{4N}
\int_0^{z^*+\varepsilon}\left[\frac{1}{\zeta'}\left(\frac{1}{\zeta'}\theta'(z)\right)'\right]'c_1^{(3)}\zeta^{1/2}dz,
\end{equation}
and
\begin{equation}\label{e531}
A_2=\frac{3}{4N}\int_0^{z^*+\varepsilon}\left[\frac{1}{\zeta'}\left(\frac{1}{\zeta'}\theta'(z)\right)'\right]'
(F_3(\zeta)-c_1^{(3)}\zeta^{1/2}) dz.
\end{equation}
Here the constant $c_1^{(3)}$ in \eqref{e530} and \eqref{e531} is the leading coefficient of $F_3(\zeta)$ defined later. This function $F_3(\zeta)$ satisfies the following relation
\begin{equation} \label{e518}
  F_2'(\zeta)=F_1(\zeta), \qquad F_3'(\zeta)=F_2(\zeta).
\end{equation}
Again, since $\zeta$ is of order $N^{2/3}$ (cf. (\ref{zeta-def})), $\hat e_1$ is also independent of $N$. To get the approximation of $A_2$, we recall the relations among $F_i(\zeta)$ in \eqref{e517} and \eqref{e518}. Define $F_1(\zeta)$ as follows
\begin{equation}\label{e563}
F_1(\zeta):=-\int_\zeta^\infty (F_0(s)-c_0s^{1/2})ds.
\end{equation}
Since the integrand in the above integral behaves like $O(\zeta^{-5/2})$ as $\zeta \to \infty$, then $F_1(\zeta)$ exists and satisfies
\eqref{e517}. Moreover, according to Lemma \ref{lemma-exp}, it satisfies the following asymptotic expansion as $\zeta \to
\infty$
\begin{equation}
F_1(\zeta)\sim\sum_{i=1}^\infty
c_i^{(1)}\zeta^{-3/2-3(i-1)}+\sum_{i=1}^\infty \bar c_i^{(1)}
\zeta^{-3/2-3(i-1)} \sin(2\zeta^{3/2})+\sum_{i=1}^\infty \tilde
c_i^{(1)} \zeta^{-3i}\cos (2\zeta^{3/2}),
\end{equation}
where $c_1^{(1)}=-2c_1/3$. Similarly as \eqref{e563}, we define
\begin{equation}
F_2(\zeta):=-\int_\zeta^\infty F_1(s)ds \quad \textrm{and} \quad F_3(\zeta):=\int_\zeta^\infty(F_2(s)-c_1^{(2)}s^{-1/2})ds-2c_1^{(2)}\zeta^{1/2}.
\end{equation}
These two functions exist and satisfy the following expansions
\begin{equation}
F_2(\zeta)\sim\sum_{i=1}^\infty
c_i^{(2)}\zeta^{-1/2-3(i-1)}+\sum_{i=1}^\infty \bar c_i^{(2)}
\zeta^{-1/2-3i}\sin(2\zeta^{3/2})+\sum_{i=1}^\infty \tilde c_i^{(2)}
\zeta^{-2-3(i-1)}\cos (2\zeta^{3/2}),
\end{equation}
\begin{equation}
F_3(\zeta)\sim\sum_{i=1}^\infty
c_i^{(3)}\zeta^{1/2-3(i-1)}+\sum_{i=1}^\infty \bar c_i^{(3)}
\zeta^{-5/2-3(i-1)}\sin(2\zeta^{3/2})+\sum_{i=1}^\infty \tilde
c_i^{(3)} \zeta^{-1-3i}\cos (2\zeta^{3/2}).
\end{equation}
As $F_3(\zeta)-c_1^{(3)}\zeta^{1/2}$ is uniformly bounded, we get the $|A_2| \leq CN^{-7/3}$, which completes the proof.
\end{proof}
With the above preparations, we are ready for the proof of Lemma \ref{main-lemma}.

\bigskip

\noindent\emph{Proof of Lemma \ref{main-lemma}.} We prove this result by induction.  For a positive integer $k$, define $F_{3k+1}(\zeta)$, $F_{3k+2}(\zeta)$ and $F_{3k+3}(\zeta)$ as
follows:
\begin{equation}\label{e569}
F_{3k+1}(\zeta) =-\int_\zeta^\infty(F_{3k}(s)-c_1^{(3k)}s^{1/2})ds,
\quad F_{3k+2}(\zeta) =-\int_\zeta^\infty F_{3k+1}(s)ds
\end{equation}
and
\begin{equation}
F_{3k+3}(\zeta)=-\int_\zeta^\infty(F_{3k+2}(s)-c_1^{(3k+2)}s^{-1/2})ds-2c_1^{(3k+2)}\zeta^{1/2}.
\end{equation}
Assume these functions satisfy the following asymptotic expansions for all integers $k\leq j-1$ as $\zeta \to \infty$:
\begin{equation}\label{f-3k+1}
F_{3k+1}(\zeta) \sim \zeta^{1/2}\sum_{i=1}^\infty
c_i^{(3k+1)}\zeta^{1-3i}+f_S^{(3k+1)}\sin(2\zeta^{3/2})+f_C^{(3k+1)}\cos(2\zeta^{3/2}),
\end{equation}
\begin{align}
&{\rm{for}}\hbox{ } 3k+1\hbox{ }{\rm even}:
\begin{cases}
f_S^{(3k+1)}=\zeta^{-(3k+6)/2}\sum_{i=1}^\infty \bar c_i^{(3k+1)}\zeta^{-3(i-1)},\\
f_C^{(3k+1)}=\zeta^{-(3k+3)/2}\sum_{i=1}^\infty \tilde c_i^{(3k+1)}\zeta^{-3(i-1)},\\
\end{cases}
\notag\\
&{\rm{for}}\hbox{ } 3k+1\hbox{ }{\rm odd}:
\begin{cases}
f_S^{(3k+1)}=\zeta^{-(3k+3)/2}\sum_{i=1}^\infty \bar c_i^{(3k+1)}\zeta^{-3(i-1)},\\
f_C^{(3k+1)}=\zeta^{-(3k+6)/2}\sum_{i=1}^\infty \tilde c_i^{(3k+1)}\zeta^{-3(i-1)};\\
\end{cases}
\end{align}
\begin{equation}
F_{3k+2}(\zeta)\sim\zeta^{1/2}\sum_{i=1}^\infty
c_i^{(3k+2)}\zeta^{-1-3(i-1)}+f_S^{(3k+2)}\sin(2\zeta^{3/2})+f_C^{(3k+2)}\cos(2\zeta^{3/2}),
\label{f-3k+2}
\end{equation}
\begin{align}
&{\rm{for}}\hbox{ } 3k+2\hbox{ }{\rm even}:
\begin{cases}
f_S^{(3k+2)}=\zeta^{-(3k+7)/2}\sum_{i=1}^\infty \bar c_i^{(3k+2)}\zeta^{-3(i-1)},\\
f_C^{(3k+1)}=\zeta^{-(3k+4)/2}\sum_{i=1}^\infty \tilde c_i^{(3k+2)}\zeta^{-3(i-1)},\\
\end{cases}
\notag\\
&{\rm{for}}\hbox{ } 3k+2\hbox{ }{\rm odd}:
\begin{cases}
f_S^{(3k+2)}=\zeta^{-(3k+4)/2}\sum_{i=1}^\infty \bar c_i^{(3k+2)}\zeta^{-3(i-1)},\\
f_C^{(3k+2)}=\zeta^{-(3k+7)/2}\sum_{i=1}^\infty \tilde c_i^{(3k+2)}\zeta^{-3(i-1)};\\
\end{cases}
\end{align}
\begin{equation}\label{f-3k+3}
F_{3k+3}(\zeta)\sim\zeta^{1/2}\sum_{i=1}^\infty
c_i^{(3k+3)}\zeta^{-3(i-1)}+f_S^{3k+3}\sin(2\zeta^{3/2})+f_C^{3k+3}\cos(2\zeta^{3/2}),
\end{equation}
\begin{align}
&{\rm{for}}\hbox{ } k\hbox{ }{\rm even}:
\begin{cases}
f_S^{(3k+3)}=\zeta^{-(3k+5)/2}\sum_{i=1}^\infty \bar c_i^{(3k+3)}\zeta^{-3(i-1)},\\
f_C^{(3k+3)}=\zeta^{-(3k+8)/2}\sum_{i=1}^\infty \tilde c_i^{(3k+3)}\zeta^{-3(i-1)},\\
\end{cases}
\notag\\
&{\rm{for}}\hbox{ } k\hbox{ }{\rm odd}:
\begin{cases}
f_S^{(3k+3)}=\zeta^{-(3k+8)/2}\sum_{i=1}^\infty \bar c_i^{(3k+3)}\zeta^{-3(i-1)},\\
f_C^{(3k+3)}=\zeta^{-(3k+5)/2}\sum_{i=1}^\infty \tilde c_i^{(3k+3)}\zeta^{-3(i-1)}.\\
\end{cases}
\end{align}
Moreover, suppose we have the following formula
\begin{equation}\label{e540}
-\frac{3}{4N}\int_0^{z^*+\varepsilon}\theta(z)F_0(\zeta)\zeta'dz=\sum_{i=0}^j
N^{-2i}\hat e_i+A_{j+1},
\end{equation}
where $\hat e_k$ are all independent of $N$ and
\begin{equation}
A_{j+1}=(-1)^{j+1}\frac{3}{4N}\int_0^{z^*+\varepsilon}\frac{d}{dz}\left(\frac{1}{\zeta'}\frac{d}{dz}\left(\frac{1}{\zeta'}\cdots\frac{d}{dz}
\left(\frac{\theta'(z)}{\zeta'}\right)\cdots\right)\right)
(F_{3j}(\zeta)-c_1^{(3j)}\zeta^{1/2})dz.
\end{equation}
Here the differential operator
\begin{equation}\label{e542}
\frac{d}{dz}\left(\frac{1}{\zeta'}\right)
\end{equation}
appears $(3j-1)$ times in the nested set of derivatives appearing in the integral above. According to the asymptotic expansion for $F_{3j}(\zeta)$, it is easily seen that $F_{3j}(\zeta)-c_1^{(3j)}\zeta^{1/2}$ is uniformly bounded. This, together with \eqref{zeta-def}, means that the error term in \eqref{e540} satisfies $|A_{j+1}|  \leq CN^{-(2j+\frac{1}{3})}$.

Now let us apply integration by parts again to derive next term in \eqref{e540}. According to definitions of $F_k(\zeta)$ in \eqref{e569}, it is easy to see that
\begin{equation}
F_{3j}(\zeta)-c_1^{(3j)}\zeta^{1/2}=\frac{1}{\zeta'}\frac{d}{dz}F_{3j+1}(\zeta),
\end{equation}
Then integration by parts gives us
\begin{equation}
A_{j+1}=(-1)^{j+1}\frac{3}{4N}\int_0^{z^*+\varepsilon}\frac{d}{dz}\left(\frac{1}{\zeta'}\frac{d}{dz}\left(\frac{1}{\zeta'}\cdots\frac{d}{dz}
\left(\frac{\theta'(z)}{\zeta'}\right)\cdots\right)\right)F_{3j+1}dz,
\end{equation}
where now the differential operator (\ref{e542}) appears $3j$ times. Two more integration by parts gives us
\begin{equation}\label{e545}
A_{j+1}=(-1)^{j+1}\frac{3}{4N}\int_0^{z^*+\varepsilon}\frac{d}{dz}\left(\frac{1}{\zeta'}\frac{d}{dz}\left(\frac{1}{\zeta'}\cdots\frac{d}{dz}
\left(\frac{\theta'(z)}{\zeta'}\right)\cdots\right)\right)F_{3j+3}dz,
\end{equation}
where now the differential operator (\ref{e542}) appears $(3j+2)$
times in the above integral. According to Lemma \ref{lemma-exp}, one can verify that the expansions in \eqref{f-3k+1}, \eqref{f-3k+2} and \eqref{f-3k+3} are also valid for $k=j$. Therefore, using \eqref{f-3k+3} for $k=j$, we split (\ref{e545}) into the following two terms
\begin{equation}
A_{j+1}=N^{-2j-2}\hat e_{j+1}+A_{j+2},
\end{equation}
where
\begin{equation}
\hat e_{j+1}=(-1)^{j+2}N^{2j+1}\int_0^{z^*+\varepsilon}\frac{d}{dz}\left(\frac{1}{\zeta'}\frac{d}{dz}\left(\frac{1}{\zeta'}\cdots\frac{d}{dz}
\left(\frac{\theta'(z)}{\zeta'}\right)\cdots\right)\right)(c_1^{(3j+3)}\zeta^{1/2})dz
\end{equation}
and
\begin{equation}
A_{j+2}=(-1)^{j+2}\frac{3}{4N}\int_0^{z^*+\varepsilon}\frac{d}{dz}\left(\frac{1}{\zeta'}\frac{d}{dz}\left(\frac{1}{\zeta'}\cdots\frac{d}{dz}
\left(\frac{\theta'(z)}{\zeta'}\right)\cdots\right)\right)(F_{3j+3}-c_1^{(3j+3)}\zeta^{1/2})dz.
\end{equation}
Here the differential operator (\ref{e542}) appears $(3j+2)$ times
in each of the above integrals. Recalling the definition of $\zeta$ in (\ref{zeta-def}) again, one can see that $\hat e_{j+1}$ is a constant
independent of $N$. Moreover, we have $|A_{j+2}|  \leq CN^{-(2j + \frac{7}{3})}$.

Therefore, we have shown that \eqref{e540} also holds when we increase $j$ by 1. As a consequence, (\ref{e510})
has an asymptotic expansion in powers of $1/N^2$ and our lemma is proved.  \hfill $\Box$

\bigskip

Now we can prove our main results.

\bigskip

\noindent\emph{Proof of Theorem \ref{thm2}.} Let us recall \eqref{rho-exp-full}. As we mentioned earlier, based on the results in \cite{em2003}, the second term of its right-hand side satisfies an asymptotic expansion in powers of $1/N^2$, whose coefficients are analytic functions of $\textbf{t}$. According to Proposition \ref{prop-four-int} and Lemma \ref{main-lemma}, the first term of the right-hand side also satisfies an asymptotic expansion in powers of $1/N^2$. The analytic property of coefficients in the expansion follows from the similar analysis as in \cite{em2003}. Adding these two terms together, we prove our theorem.  \hfill $\Box$

\begin{rmk}
  As mentioned at the beginning of this paper, a combination of \eqref{zn-int} and Theorem \ref{thm2} gives us Theorem \ref{main-thm}.
\end{rmk}

\subsection*{Acknowledgements}
  D. Dai was partially supported by a grant from City University of Hong Kong (Project No. 7002883) and a grant from the Research Grants Council of the Hong Kong Special Administrative Region, China (Project No. CityU 101411).

%%%%%%%%%%%%%%%%%%%%%%%%%%%%%%%%%%%%%%%%%%%%%%%%%%%%%%%%%%%

\end{document}